\documentclass[a4paper,10pt]{amsart}

\usepackage[utf8]{inputenc}
\usepackage[T2A]{fontenc}
\usepackage{etoolbox}
\usepackage{tikz-cd}

\makeatletter
\let\old@tocline\@tocline
\let\section@tocline\@tocline
\newcommand{\subsection@dotsep}{4.5}
\newcommand{\subsubsection@dotsep}{4.5}
\patchcmd{\@tocline}
{\hfil}
{\nobreak
	\leaders\hbox{$\m@th
		\mkern \subsection@dotsep mu\hbox{.}\mkern \subsection@dotsep mu$}\hfill
	\nobreak}{}{}
\let\subsection@tocline\@tocline
\let\@tocline\old@tocline

\patchcmd{\@tocline}
{\hfil}
{\nobreak
	\leaders\hbox{$\m@th
		\mkern \subsubsection@dotsep mu\hbox{.}\mkern \subsubsection@dotsep mu$}\hfill
	\nobreak}{}{}
\let\subsubsection@tocline\@tocline
\let\@tocline\old@tocline

\let\old@l@subsection\l@subsection
\let\old@l@subsubsection\l@subsubsection

\def\@tocwriteb#1#2#3{%
	\begingroup
	\@xp\def\csname #2@tocline\endcsname##1##2##3##4##5##6{%
		\ifnum##1>\c@tocdepth
		\else \sbox\z@{##5\let\indentlabel\@tochangmeasure##6}\fi}%
	\csname l@#2\endcsname{#1{\csname#2name\endcsname}{\@secnumber}{}}%
	\endgroup
	\addcontentsline{toc}{#2}%
	{\protect#1{\csname#2name\endcsname}{\@secnumber}{#3}}}%

\newlength{\@tocsectionindent}
\newlength{\@tocsubsectionindent}
\newlength{\@tocsubsubsectionindent}
\newlength{\@tocsectionnumwidth}
\newlength{\@tocsubsectionnumwidth}
\newlength{\@tocsubsubsectionnumwidth}
\newcommand{\settocsectionnumwidth}[1]{\setlength{\@tocsectionnumwidth}{#1}}
\newcommand{\settocsubsectionnumwidth}[1]{\setlength{\@tocsubsectionnumwidth}{#1}}
\newcommand{\settocsubsubsectionnumwidth}[1]{\setlength{\@tocsubsubsectionnumwidth}{#1}}
\newcommand{\settocsectionindent}[1]{\setlength{\@tocsectionindent}{#1}}
\newcommand{\settocsubsectionindent}[1]{\setlength{\@tocsubsectionindent}{#1}}
\newcommand{\settocsubsubsectionindent}[1]{\setlength{\@tocsubsubsectionindent}{#1}}

\renewcommand{\l@section}{\section@tocline{1}{\@tocsectionvskip}{\@tocsectionindent}{}{\@tocsectionformat}}%
\renewcommand{\l@subsection}{\subsection@tocline{2}{\@tocsubsectionvskip}{\@tocsubsectionindent}{}{\@tocsubsectionformat}}%
\renewcommand{\l@subsubsection}{\subsubsection@tocline{3}{\@tocsubsubsectionvskip}{\@tocsubsubsectionindent}{}{\@tocsubsubsectionformat}}%
\newcommand{\@tocsectionformat}{}
\newcommand{\@tocsubsectionformat}{}
\newcommand{\@tocsubsubsectionformat}{}
\expandafter\def\csname toc@1format\endcsname{\@tocsectionformat}
\expandafter\def\csname toc@2format\endcsname{\@tocsubsectionformat}
\expandafter\def\csname toc@3format\endcsname{\@tocsubsubsectionformat}
\newcommand{\settocsectionformat}[1]{\renewcommand{\@tocsectionformat}{#1}}
\newcommand{\settocsubsectionformat}[1]{\renewcommand{\@tocsubsectionformat}{#1}}
\newcommand{\settocsubsubsectionformat}[1]{\renewcommand{\@tocsubsubsectionformat}{#1}}
\newlength{\@tocsectionvskip}
\newcommand{\settocsectionvskip}[1]{\setlength{\@tocsectionvskip}{#1}}
\newlength{\@tocsubsectionvskip}
\newcommand{\settocsubsectionvskip}[1]{\setlength{\@tocsubsectionvskip}{#1}}
\newlength{\@tocsubsubsectionvskip}
\newcommand{\settocsubsubsectionvskip}[1]{\setlength{\@tocsubsubsectionvskip}{#1}}

\patchcmd{\tocsection}{\indentlabel}{\makebox[\@tocsectionnumwidth][l]}{}{}
\patchcmd{\tocsubsection}{\indentlabel}{\makebox[\@tocsubsectionnumwidth][l]}{}{}
\patchcmd{\tocsubsubsection}{\indentlabel}{\makebox[\@tocsubsubsectionnumwidth][l]}{}{}

\newcommand{\@sectypepnumformat}{}
\renewcommand{\contentsline}[1]{%
	\expandafter\let\expandafter\@sectypepnumformat\csname @toc#1pnumformat\endcsname%
	\csname l@#1\endcsname}
\newcommand{\@tocsectionpnumformat}{}
\newcommand{\@tocsubsectionpnumformat}{}
\newcommand{\@tocsubsubsectionpnumformat}{}
\newcommand{\setsectionpnumformat}[1]{\renewcommand{\@tocsectionpnumformat}{#1}}
\newcommand{\setsubsectionpnumformat}[1]{\renewcommand{\@tocsubsectionpnumformat}{#1}}
\newcommand{\setsubsubsectionpnumformat}[1]{\renewcommand{\@tocsubsubsectionpnumformat}{#1}}
\renewcommand{\@tocpagenum}[1]{%
	\hfill {\mdseries\@sectypepnumformat #1}}

\let\oldappendix\appendix
\renewcommand{\appendix}{%
	\leavevmode\oldappendix%
	\addtocontents{toc}{%
		\protect\settowidth{\protect\@tocsectionnumwidth}{\protect\@tocsectionformat\sectionname\space}%
		\protect\addtolength{\protect\@tocsectionnumwidth}{2em}}%
}
\makeatother

\makeatletter
\settocsectionnumwidth{2em}
\settocsubsectionnumwidth{2.5em}
\settocsubsubsectionnumwidth{3em}
\settocsectionindent{1pc}%
\settocsubsectionindent{\dimexpr\@tocsectionindent+\@tocsectionnumwidth}%
\settocsubsubsectionindent{\dimexpr\@tocsubsectionindent+\@tocsubsectionnumwidth}%
\makeatother

\settocsectionvskip{10pt}
\settocsubsectionvskip{0pt}
\settocsubsubsectionvskip{0pt}

\settocsectionformat{\bfseries}
\settocsubsectionformat{\mdseries}
\settocsubsubsectionformat{\mdseries}
\setsectionpnumformat{\bfseries}
\setsubsectionpnumformat{\mdseries}
\setsubsubsectionpnumformat{\mdseries}

\let\oldtableofcontents\tableofcontents
\renewcommand{\tableofcontents}{%
	\vspace*{-\linespacing}% Default gap to top of CONTENTS is \linespacing.
	\oldtableofcontents}

\usepackage{xpatch}
\usepackage{soul}
\usepackage{amssymb,latexsym,amsmath,amsthm,graphicx,bm}
\usepackage{sidecap,multicol}
\usepackage[all,cmtip]{xy}
\usepackage[pdfencoding=auto]{hyperref}
\hypersetup{
    linktoc=all,     %set to all if you want both sections and subsections linked
}
\usepackage[capitalize,nameinlink]{cleveref}

\usepackage[shortlabels]{enumitem}

\usepackage{tikz}
\usetikzlibrary{arrows.meta,decorations.pathmorphing,decorations.markings,backgrounds,positioning,fit,shapes}

\newcommand\R{\mathbb{R}}
\newcommand\CC{\mathbb{C}}

\newcommand\N{\mathbb{N}}

\newcommand{\Int}{\operatorname{int}}
\newcommand{\Pol}{\operatorname{Pol}}

\newcommand{\JC}{\mathfrak{J}}
\newcommand{\IC}{\mathfrak{T}}

\newcommand{\HS}{\mathrm{H}}

\newcommand{\CS}{\mathrm{S}}
\newcommand{\CU}{{}^+\mathrm{S}}
\newcommand{\RU}{\underline{\R}}
\newcommand{\kU}{\underline{k}}
\newcommand{\OmegaU}{{}^+\Omega}

\newcommand{\FC}{\mathcal{F}}
\newcommand{\GC}{\mathcal{G}}
\newcommand{\supp}{\mathrm{supp}\,}
\newcommand{\Id}{\mathrm{id}}
\newcommand{\Lip}{\mathrm{Lip}}
\newcommand{\sm}{\mathrm{sm}}
\newcommand{\res}{\mathrm{res}}
\newcommand{\relint}{\mathrm{relint}\,}
\newcommand{\Ker}{\mathrm{Ker}\,}
\newcommand{\PW}{\mathrm{PPol}}
\newcommand{\HH}{\mathbb{H}}

\newcommand{\pp}{\mathfrak{p}}

\newcommand{\sh}{\mathrm{sh}}

\newcommand{\spec}{\mathrm{spec} _{\mathbb{R}}}
\newcommand{\Aff}{\mathrm{aff}}

\def\colim{\underset{\longrightarrow}{\operatorname{lim\,}}}

\newcommand{\mh}{}
\mathchardef\mh="2D

\DeclareMathOperator{\sheafHom}{\mathcal{H}\kern -1.2pt \mathit{om}}

\newtheorem{theorem}{Theorem}
\newtheorem{proposition}[theorem]{Proposition}
\newtheorem{corollary}[theorem]{Corollary}
\newtheorem{lemma}[theorem]{Lemma}

\newtheorem*{theorem*}{Theorem}
\theoremstyle{definition}
\newtheorem*{definition}{Definition}

\theoremstyle{remark}
\newtheorem*{remark}{Remark}

\usepackage[maxbibnames=10, maxcitenames=10, backend=biber, bibencoding=utf8]{biblatex}
\renewbibmacro{in:}{}
\DeclareFieldFormat{volume}{\bibstring{jourvol}\addnbspace #1}
\DeclareFieldFormat{number}{\bibstring{number}\addnbspace #1}
\renewbibmacro*{journal+issuetitle}{%
  \usebibmacro{journal}%
  \setunit*{\addcomma\space}% was: \setunit*{\addspace}
  \iffieldundef{series}
    {}
    {\newunit
     \printfield{series}%
     \setunit{\addcomma\space}}% was: \setunit{\addspace}
  \usebibmacro{volume+number+eid}%
  \setunit{\addspace}%
  \usebibmacro{issue+date}%
  \setunit{\addcolon\space}%
  \usebibmacro{issue}%
  \newunit}
\title{The de Rham cohomology\\ of soft function algebras}
\author{Igor Baskov}
\thanks{This research was supported by the Ministry of Science and Higher Education of the Russian Federation, agreement 075-15-2019-1620 date 08/11/2019 and 075-15-2022-289 date 06/04/2022.
}

\begin{document}
% tikz stuff
\tikzset{>={Stealth[scale=1.2]}}
\tikzset{->-/.style={decoration={
			markings,
			mark=at position #1 with {\arrow{>}}},postaction={decorate}}}
\tikzset{-w-/.style={decoration={
			markings,
			mark=at position #1 with {\arrow{Stealth[fill=white,scale=1.4]}}},postaction={decorate}}}
\tikzset{->-/.default=0.65}
\tikzset{-w-/.default=0.65}
\tikzstyle{bullet}=[circle,fill=black,inner sep=0.5mm]
\tikzstyle{circ}=[circle,draw=black,fill=white,inner sep=0.5mm]
\tikzstyle{vertex}=[circle,draw=black,thick,inner sep=0.5mm]
\tikzset{darrow/.style={double distance = 4pt,>={Implies},->},
	darrowthin/.style={double equal sign distance,>={Implies},->},
	tarrow/.style={-,preaction={draw,darrow}},
	qarrow/.style={preaction={draw,darrow,shorten >=0pt},shorten >=1pt,-,double,double
		distance=0.2pt}}

\newcommand{\Addresses}{{% additional braces for segregating \footnotesize
  \bigskip
  \footnotesize
\noindent
  \textsc{St. Petersburg Department of Steklov Mathematical Institute\\of Russian Academy of Sciences}\par\noindent
  \textit{E-mail address}: \texttt{baskovigor@pdmi.ras.ru}
}}

\begin{abstract}
We study the dg-algebra $\Omega ^\bullet _{A|\R}$ of algebraic de Rham forms of a real soft function algebra $A$, i.e., the algebra of global sections of a soft subsheaf of $C_X$, the sheaf of continuous functions on a space $X$.
We obtain a canonical splitting $\HS ^n (\Omega ^\bullet _{A|\R}) \cong \HS ^n (X,\R)\oplus V$, where $V$ is some vector space.
In particular, we consider the cases $A=C(X)$ for $X$ a compact Hausdorff space and $A = C^\infty (X)$ for $X$ a compact smooth manifold.
For the algebra $\PW _K (|K|)$ of piecewise polynomial functions on a polyhedron $K$ the above splitting reduces to a canonical isomorphism $\HS ^* (\Omega ^\bullet _{\PW _K (|K|)|\R}) \cong \HS ^* (|K|,\R)$.
We also prove that the algebraic de Rham cohomology $\HS ^n (\Omega ^\bullet _{C(X)|\R})$ is nontrivial for each $n\geq 1$ if $X$ is an infinite compact Hausdorff space.

\end{abstract}

\maketitle
\tableofcontents
\section{Introduction}

All algebras are assumed to be commutative and all dg-algebras are assumed to be graded-commutative.
To an algebra $A$ over a field $k$ one associates a dg-algebra $\Omega ^\bullet _{A|k}$ with $\Omega ^0 _{A|k} = A$, called the de Rham dg-algebra, in a standard way (see Subsection~\ref{ssec:algebraicderhamforms} or \cite[$\S 3$]{KunzKahlerDifferentials}).
Our main focus will be the cases of $\R$-algebras $A=C(X)$ for a compact Hausdorff space $X$, $A=C^\infty (M)$ for a smooth manifold $M$ (possibly with boundary) and $A=\PW _K (|K|)$, the algebra of piecewise polynomial functions on a polyhedron $K$.
Here all functions are assumed to be real-valued, but our results will hold for complex-valued functions also.
We study the cohomology groups $\HS ^* (\Omega ^\bullet _{A|k})$.

There is a canonical morphism of dg-algebras $\pi :\Omega ^\bullet _{C^\infty (M)|\R}\to \Omega ^\bullet (M)$, where $\Omega ^\bullet (M)$ is the dg-algebra of smooth differential $\R$-forms on $M$ (see Subsection~\ref{sssec:pimorphism}).
The morphism $\pi$ is the identity in degree $0$.
It is not an isomorphism.
For example, the equality $df(t)=(\partial f /\partial t)dt$ holds in $\Omega ^1 _{C^\infty (\R )|\R}$ if and only if  $f(t)$ is an algebraic function of $t$ (\cite[Corollary of Proposition~$1$]{OsbornDerivations}).
Moreover, if $M$ is of dimension $\geq 1$, then the cardinality of any set of generators for $\Omega ^1 _{C^\infty (M)|\R}$ as a $C^\infty (M)$-module is at least that of $\R$ (\cite[Corollary~$15$]{gomez1990number}).
The map on cohomology, induced by $\pi$, is not an isomorphism either, see \cite{OsbornDerivations} for the proof that the closed form $dt/(1+t^2)$ is not exact in $\Omega ^\bullet _{C^\infty (\R )|\R}$.
On the other hand, consider the algebra $A$ of regular functions on a smooth affine variety $V$ over $\CC$.
It is a subalgebra of $C^\infty (V, \CC)$ with inclusion $i:A\hookrightarrow C^\infty (V, \CC)$.
Consider the composition $$\Omega ^\bullet _{A|\CC} \xrightarrow{\Omega _i} \Omega ^\bullet _{C^\infty (V)|\CC}\xrightarrow{\pi _\CC}\Omega ^\bullet (V,\CC ),$$
where $\pi _\CC$ is the analogue of $\pi$ over $\CC$.
Grothendieck's comparison theorem states that the induced map $\HS ^* (\Omega ^\bullet _{A|\CC})\to \HS ^* (\Omega ^\bullet (V,\CC ))$ is an isomorphism (\cite[Theorem~$1^\prime$]{grothendieck1966rham}).

For a soft sheaf of $k$-algebras $\FC$ on a compact Hausdorff space $X$ we construct (see Section~\ref{sec:theoppositemaplambda}) a linear map
$$\Lambda _\FC :\HH ^* (X,\kU _X [0])\to \HS ^* (\Omega ^\bullet _{\FC (X)|k}).$$
Here the domain is the cohomology of $X$ with coefficients in the constant sheaf $\kU _X$.
This map is natural with respect to morphisms of ringed spaces (see Proposition~\ref{prp:lambdaisfunctorial}).
We are mostly interested in the case $k=\R$.

We prove (see Theorem~\ref{thm:thecompositionforsmoothfunctions}) that for the sheaf of smooth functions $C^\infty _M$ on a smooth manifold $M$ the following diagram is commutative:
\begin{equation*}
\begin{tikzcd}
\HH ^* (M,\RU _M [0])\arrow[rr,"\Theta",bend left=10]\arrow[r,"\Lambda _{C^\infty _M}"']&\HS ^*(\Omega ^\bullet _{C ^\infty (M)|\R})\arrow[r,"\HS (\pi )"']&\HS ^*(\Omega ^\bullet (M)),
\end{tikzcd}
\end{equation*}
where $\Theta$ is the canonical isomorphism (see Subsection~\ref{sssec:thetamorphism}).
In particular, the map
$$\HS (\pi ):\HS ^n (\Omega ^\bullet _{C ^\infty (M)|\R})\to \HS ^n (\Omega ^\bullet (M))$$
is surjective for $n\geq 0$.
This is a generalization of a result obtained by Gómez, namely that $\HS (\pi )$ is surjective for $n$ even (see \cite[4.Theorem]{Gomez1992}).

Next, for an arbitrary space $X$ and a subalgebra $i:A\hookrightarrow C(X)$, we construct a linear map (see Section~\ref{sec:theconstructionofpsi})
$$\Psi _A:\HS ^* (\Omega ^\bullet _{A|\R})\to \HH ^* (X,\RU _X [0]).$$
Our construction of $\Psi _A$ relies on local Lipschitz contractibility of algebraic sets (Theorem~\ref{thm:localcontractabilityofsemialgebraicsets}) due to Shartser (\cite[Theorem~$4.18$]{shartser2011rham}).
This map is natural with respect to continuous maps of spaces covered by a homomorphism of algebras (see Proposition~\ref{prp:functorialityofpsiinspaces}); in particular, $\Psi _A=\Psi _{C(X)}\circ\Omega _i$.
We prove (see Theorem~\ref{thm:compositionforsoftsheaf}) that for a compact Hausdorff space $X$ and a soft subsheaf $\FC$ of $C_X$, the sheaf of continuous functions, the composition $\Psi _{\FC (X)} \circ\Lambda _\FC$ coincides with the identity map.
Thus, the groups $\HH ^* (X,\RU _X [0])$ canonically split off of $\HS ^* (\Omega ^\bullet _{\FC (X)})$.

We also check (see Theorem~\ref{thm:compositionproposition}) that for a smooth manifold $M$ the following diagram is commutative:
\begin{equation*}
\begin{tikzcd}
\HS ^* (\Omega ^\bullet _{C^\infty (M)|\R})\arrow[r,"\Psi _{C^\infty (M)}"]\arrow[rd,"\HS (\pi)"']&\HH ^* (M, \RU _M [0])\arrow[d,"\Theta"]\\
&\HS ^* (\Omega ^\bullet (M)).
\end{tikzcd}
\end{equation*}

For the sheaf of piecewise polynomial functions $\PW _K$ on a polyhedron $K$ (see Section~\ref{sec:piecewisepolynomialfunctions}) we prove that the morphisms $\Lambda _{\PW _K}$ and $\Psi _{\PW _K (|K|)}$ are isomorphisms (see Theorem~\ref{thm:lambdaandpsiforpiecewisepolynomialfunctions}).

In Section~\ref{sec:themapspsiandlambdaforsoftfunctionalgebras} we describe the group $\HS ^0 (\Omega ^\bullet _{A|k} )$ for a general function algebra $A$ over a field $k$ of characteristic $0$ (Corollary~\ref{cor:locallyconstantfunctions}) calculated by G{\'o}mez in \cite{gomez1990number}.
The related result is \cite[Proposition~$5$]{OsbornDerivations} (note that the cohomology considered there is, in general, different from ours).
We show that for a soft subsheaf of algebras $\FC\subset C_X$ on a compact Hausdorff space $X$ the morphisms $\Lambda _\FC$ and $\Psi _{\FC (X)}$ are isomorphisms in degree $0$ (Proposition~\ref{prp:lambdaisisomorphismindegree0} and Corollary~\ref{cor:psiisisoindegree0}).
We also prove that for an infinite compact Hausdorff space $X$ the maps $\Lambda _{C_X}: \HH ^n (X,\RU _X [0])\to \HS ^n (\Omega ^\bullet _{C(X)|\R})$ and $\Psi _{C(X)}:\HS ^n (\Omega ^\bullet _{C(X)|\R})\to \HH ^n (X,\RU _X [0])$ are not isomorphisms in degrees $n>0$.
The same is true for the algebra of smooth functions on a smooth manifold, see Subsection~\ref{ssec:lambdaandpsiarenotiso}.

From our results one can deduce the similar results for $\CC$-algebras.

In the paper we only consider algebraic structures, however one can consider topological algebras.
Using projective tensor products, one can define the dg-algebra $\tilde{\Omega} ^\bullet _{\mathcal{A}|\CC}$ for a Fr{\'e}chet $\CC$-algebra $\mathcal{A}$ (denoted by $\Omega ^\bullet _{\mathrm{ab}}\mathcal{A}$ in \cite[$\S 8$]{gracia2013elements}).
This dg-algebra is a topological analogue of the de Rham dg-algebra.

Consider the Fr{\'e}chet algebra $\mathcal{A}=C^\infty (M,\CC )$ for a compact smooth manifold $M$.
The dg-algebra $\tilde{\Omega} ^\bullet _{\mathcal{A}|\CC}$ is isomorphic to $\Omega ^\bullet (M,\CC)$, see \cite[Proposition~$8.1$]{gracia2013elements}.

Consider the Banach algebra $\mathcal{A}=C(X,\CC )$ for a compact Hausdorff space $X$.
Then $\tilde{\Omega} ^n _{\mathcal{A}|\CC} = 0$ for $n\geq 1$.
To see that, first note that by \cite[$\S 8$]{johnson1972cohomology} or \cite[Remarks~$47$, d]{connes1985non}, the continuous Hochschild cohomology $\HS\HS ^n(\mathcal{A},\mathcal{A}^*)$ is zero for positive $n$.
Then, by \cite[Corollary~$1.3$]{johnson1972cohomology} the continuous Hochschild homology of $\mathcal{A}$ is zero in positive degrees, in particular, $\HS\HS _1(\mathcal{A}) = 0$.
By \cite[p.~346]{gracia2013elements}, we have $\HS\HS _1(\mathcal{A})\cong \tilde{\Omega} ^1 _{\mathcal{A}|\CC}$.
Hence, by the construction of $\tilde{\Omega} ^n _{\mathcal{A}|\CC}$ this space is zero for $n\geq 1$.
\subsection*{Acknowledgement}
I would like to thank Dr. Semёn Podkorytov for his patience, numerous fruitful discussions and help drafting this paper.
I am grateful to the St. Petersburg Department of Steklov Mathematical Institute of Russian Academy of Sciences for their financial assistence.

\section{Preliminaries}\label{sec:preliminaries}

\subsection{Sheaves}\label{ssec:sheaves}
Here we outline the basic definitions and facts of sheaf theory needed in the paper.
For this we follow the books of Godement \cite[]{godement1958topologie}, Wedhorn \cite[]{wedhorn2016manifolds} and Bredon \cite{bredon2012sheaf}.

We refer to \cite[Definition~$10.2$]{wedhorn2016manifolds} for the definition of the \textit{hypercohomology} groups $\HH ^* (X,\FC ^\bullet )$ of a complex of sheaves $\FC ^\bullet$ on a space $X$.
By a complex we always mean a non-negative cochain complex.
For a sheaf $\FC$ we denote by $\FC [0]$ the complex of sheaves with $\FC$ in degree $0$ and other terms zero.

Let $\FC ^\bullet$ be a complex of sheaves on $X$.
Then there is a canonical homomorphism
$$\Upsilon :\HS ^* (\FC ^\bullet (X))\to\HH ^* (X,\FC ^\bullet),$$
natural with respect to morphisms of complexes of sheaves.
The map $\Upsilon$ is an isomorphism in degree $0$. 
If the sheaves $\FC ^n$ are acyclic then $\Upsilon$ is an isomorphism, see \cite[Theorem and Definition~$10.4$, Proposition~$10.8$]{wedhorn2016manifolds}.

\begin{lemma}\label{lem:functorialityofhypercohomology}
Take complexes of sheaves $\FC ^\bullet$ and $\GC ^\bullet$ on topological spaces $X$ and $Y$, respectively.
Suppose $f:X\to Y$ is a continuous map and $\varphi :\GC ^\bullet \to f_* \FC ^\bullet$ is a morphism of complexes of sheaves.
Then there is the induced map on hypercohomology $f^* :\HH (Y,\GC ^\bullet )\to \HH (X,\FC ^\bullet )$.
Moreover, the following diagram is commutative:
\begin{equation*}
\begin{tikzcd}
\HH ^*(X,\FC ^\bullet ) & \HS ^* (\FC ^\bullet (X))\arrow[l,"\Upsilon"] \\
\HH ^*(Y,\GC ^\bullet )\arrow[u,"f^*"] & \HS ^* (\GC ^\bullet (Y))\arrow[l,"\Upsilon"]\arrow[u,"\HS (\varphi (Y))"].
\end{tikzcd}
\end{equation*}
\end{lemma}
\begin{proof}
See \cite{stacks-project}.
\end{proof}

For a sheaf $\FC$ on $X$ and a closed subset $Z\subset X$ define
$$\FC (Z):= \underset{\text{open} \, U\supset Z}{\colim} \FC (U).$$

For a sheaf $\FC$ we fix the notation for the restriction map $\res _{W,W^\prime} :\FC (W)\to \FC (W^\prime)$ for sets $W^\prime\subset W$ open or closed.

A sheaf $\FC$ is called \textit{soft} if for every closed set $Z\subset X$ the restriction map $\res _{X,Z}$ is surjective.
Any soft sheaf on a compact Hausdorff space is acyclic \cite[Proposition~$10.17$]{wedhorn2016manifolds}.
A sheaf $\FC$ is called \textit{flabby} if for every open set $U\subset X$ the restriction map $\FC (X)\to \FC (U) $
is surjective. A flabby sheaf is acyclic.

For a presheaf $F$ denote by ${}^+F$ the sheafification of $F$.
Let $\sh :F\to {}^+F$ denote the sheafification map.

\subsection{Algebraic sets}\label{ssec:algebraicsets}

An \textit{algebraic set} in $\R ^m$ is the set of solutions of a system of polynomial equations in $\R ^m$.

\begin{definition}\label{def:realspectrum}
Take $B$ a finitely generated $\R$-algebra.
Following \cite[$3.4$]{nestruev2003smooth}, we define the \textit{real spectrum} $\spec B$ of $B$ as the set of algebra homomorphisms $\psi :B\to \R$.
\end{definition}
A set of generators $b_1,\dots ,b_n$ of $B$ gives rise to an injective map $i:\spec B\hookrightarrow \R ^m$, $i(\psi ):=(\psi (b_1),\dots ,\psi (b_n))$.
We call such a map a \textit{distinguished embedding}.
The image of a distinguished embedding is an algebraic set.
Equip $\spec B$ with the induced topology under some distinguished embedding.
This topology does not depend on the choice of a distinguished embedding because
for any two distinguished embeddings $i:\spec B\hookrightarrow \R ^m$ and $j:\spec B\hookrightarrow \R ^l$ there exists a polynomial map $p:\R ^m\to \R ^l$ such that the following diagram is commutative
\begin{equation*}
\begin{tikzcd}
\spec B\arrow[r,"i"]\arrow[rd,"j"]&\R ^m\arrow[d,"p"]\\
&\R ^l.
\end{tikzcd}
\end{equation*}
A polynomial map is smooth, hence, we can talk of maps into $\spec B$ being smooth or locally Lipschitz.
Hence, $\spec B$ becomes a \textit{predifferentiable space} (see \cite[$\S 1$]{chenpredifferentiable} for the definition).

If $\varphi:B^\prime\to B$ is a homomorphism of finitely generated $\R$-algebras, then define $\spec \varphi:\spec B \to \spec B^\prime$ as $(\spec \varphi) (\psi ):= \psi\circ\varphi$.
For $b\in B$ define the function $\hat{b}\in C(\spec B)$ by $\hat{b} (\psi ):= \psi (b)$.
We have $\widehat{\varphi(c)}=\widehat{c}\circ (\spec \varphi)$ in $C(\spec B)$ for $c\in B^\prime$.

Now, let's assume that $B$ is a finitely generated $\R$-subalgebra of $C(X)$ for some topological space $X$.
To each point $x\in X$ we associate the homomorphism $B\to \R$, $b\mapsto b(x)$.
We get a continuous map $\Gamma _B :X\to \spec B$.
If $M$ is a smooth manifold and $B$ is a finitely generated subalgebra of $C^\infty (M)$ then $\Gamma _B$ is smooth.

For $b\in B$ we have
\begin{equation}\label{eq:hatgamma}
\hat{b}\circ \Gamma _B =b.
\end{equation}

Take $f:X\to Y$ a continuous map of topological spaces.
Suppose that $B\subset C(X)$ and $B^\prime\subset C(Y)$ are finitely generated subalgebras and $\varphi :B^\prime\to B$ is a homomorphism such that the following diagram is commutative:
\begin{equation*}
\begin{tikzcd}
C(X)&B\arrow[l,hook']\\
C(Y)\arrow[u,"f^*"]&B^\prime \arrow[l,hook']\arrow[u,"\varphi"].
\end{tikzcd}
\end{equation*}
Then the following diagram is commutative:
\begin{equation*}
\begin{tikzcd}
X \arrow[r,"\Gamma _B"]\arrow[d,"f"] & \spec B \arrow[d,"\spec \varphi"] \\
Y \arrow[r,"\Gamma _{B^\prime}"] & \spec B^\prime
.\end{tikzcd}
\end{equation*}

\subsection{Sheaves of singular cochains}\label{ssec:sheavesofsingularcochains}

For a topological space $X$ let $\CS _{n} (X)$ denote the space of singular $\R$-chains in $X$ and let $\CS ^n (X)$ denote the dual space, the space of singular cochains.
We will need the following presheaves $\CS ^n _X $ for $n\geq 0$:
$$U\mapsto \CS ^n (U).$$
The standard differential on cochains gives rise to a complex of presheaves $\CS ^\bullet _X$ on $X$.
For a smooth manifold $M$ (possibly with boundary) let $\CS _{\sm ,n} (M)$ denote the space of smooth singular $\R$-chains in $M$ and let $\CS ^n _{\sm}(M)$ denote the dual space.
We will need the following presheaves $\CS ^n _{\sm ,M} $ for $n\geq 0$:
$$U\mapsto \CS ^n _{\sm} (U).$$
The standard differential on cochains gives rise to a complex of presheaves $\CS ^\bullet _{\sm ,M}$ on $M$.
More generally, the complex of presheaves $\CS ^\bullet _{\sm ,Q}$ is well defined for $Q$ a predifferentiable space and functorial with respect to smooth maps between predifferentiable spaces.
In particular, $\CS ^\bullet _{\sm ,\spec B}$ are well defined.

For a set $V\subset \R ^m$ let $\CS _{\Lip ,n} (V)$ denote the space of Lipschitz singular $\R$-chains in $V$ and let $\CS ^n _{\Lip}(V)$ denote the dual space.
We will need the following presheaves $\CS ^n _{\Lip ,V} $ for $n\geq 0$:
$$U\mapsto \CS ^n _{\Lip} (U).$$
The standard differential on cochains gives rise to a complex of presheaves $\CS ^\bullet _{\Lip ,V}$ on $V$.
The complex of presheaves $\CS ^\bullet _{\Lip ,\spec B}$ is well defined and functorial with respect to algebra homomorphisms.

The sheaves $\CU ^n _X$ are flabby by \cite[proof of Theorem~$11.13$]{wedhorn2016manifolds} and, hence, are acyclic.
Similarly, one can prove that the sheaves $\CU ^n _{\sm ,Q}$ and $\CU ^n _{\Lip ,V}$ are flabby.

For the complex of sheaves $\CU ^\bullet _X$ we define the morphism of complexes of sheaves, called coaugmentation, $\epsilon :\RU _X [0]\to \CU ^\bullet _X$ by $\epsilon (1):=1$.
For the complexes $\CU ^\bullet _{\sm ,Q}$ and $\CU ^\bullet _{\Lip ,V}$ the coaugmentation is defined similarly.

\begin{lemma}\label{lem:globalsectionssheafificationsingularcochains}
\mbox{}
\begin{enumerate}
\item
For a topological space $X$ the morphism of complexes
$$\sh :\CS ^\bullet _X (X)\to \CU ^\bullet _X (X)$$
is a quasi-isomorphism.
\item
For a predifferentiable space $Q$ the morphism of complexes
$$\sh :\CS ^\bullet _{\sm ,Q} (Q)\to \CU ^\bullet _{\sm ,Q} (Q)$$
is a quasi-isomorphism.
\item
For a subspace $V\subset\R ^m$ the morphism of complexes
$$\sh :\CS ^\bullet _{\Lip ,V} (V)\to \CU ^\bullet _{\Lip ,V} (V)$$
is a quasi-isomorphism.
\end{enumerate}
\end{lemma}
See \cite[$7,$ p. $26$]{bredon2012sheaf} for the proof of the first case.
For the other two a similar proof applies.

\begin{lemma}\label{lem:singularcochainsquasiisomorphism}
\mbox{}
\begin{enumerate}
\item
For a locally contractible $X$ the coaugmentation $\epsilon :\RU _X[0] \to \CU ^\bullet _X$ is a quasi-isomorphism.
\item
For a smooth manifold $M$ the coaugmentation $\epsilon :\RU _M[0] \to \CU ^\bullet _{\sm ,M}$ is a quasi-isomorphism.
\item
For a locally Lipschitz contractible set $V\subset\R ^m$ the coaugmentation $\epsilon :\RU _V[0] \to \CU ^\bullet _{\Lip ,V}$ is a quasi-isomorphism.
\end{enumerate}
\end{lemma}
See \cite[Example~II.$1.2$]{bredon2012sheaf} for the case of $\CU ^\bullet _X$, the rest is analogous.

An important result needed in this work is the local Lipschitz contractability of algebraic sets due to Shartser \cite[Theorem~$4.1.8$]{shartser2011rham}.

\begin{theorem}\label{thm:localcontractabilityofsemialgebraicsets}
Take an algebraic set $V\subset \R ^m$ and a point $v_0\in V$.
Then there exist an open set $U\subset \R^ m$ with $v_0\in U$ and a Lipschitz map $F:U\times [0,1]\to U$ such that
\begin{enumerate}
\item
$F(u,0)=u$ for $u\in U$;
\item
$F(u,1)=v_0$ for $u\in U$;
\item
$F(v,t)\in V$ for all $v\in V\cap U$ and $t\in [0,1]$.
\end{enumerate}
\end{theorem}

\begin{corollary}\label{cor:lipschitzcochainsresolution}
The coaugmentation $\epsilon :\RU _V [0]\to \CU ^\bullet _{\Lip ,V}$ is a quasi-isomorphism for $V$ an algebraic set in $\R ^m$.
\end{corollary}
\begin{proof}
Follows directly from Lemma~\ref{lem:singularcochainsquasiisomorphism} and Theorem~\ref{thm:localcontractabilityofsemialgebraicsets}.
\end{proof}

\subsection{The de Rham dg-algebra $\Omega ^\bullet _{A|k}$ of $A$}\label{ssec:algebraicderhamforms}

To a $k$-algebra $A$ one associates the dg-algebra $\Omega ^\bullet _{A|k}$ (\cite[Theorem~$3.2$]{KunzKahlerDifferentials}) with $\Omega ^0 _{A|k} = A$.
It has the following universal property: for any dg-algebra $E$ and any algebra homomorphism $f:A\to E^0$ there exists a unique morphism of dg-algebras $F:\Omega ^\bullet _{A|k}\to E$ such that $F|_A = f$:
\begin{equation*}
\begin{tikzcd}
A\arrow[r]\arrow[rd,"f"]&\Omega ^\bullet _{A|k}\arrow[d,dashed,"F"]\\
&E.
\end{tikzcd}
\end{equation*}
The elements of $\Omega ^n _{A|k}$ are called algebraic $n$-forms.
The dg-algebra $\Omega ^\bullet _{A|k}$ is covariant in the algebra $A$.

\begin{lemma}\label{lem:kernelofinducedmapofomegas}
Suppose $A$ and $B$ are $k$-algebras and $\varphi:A\to B$ is a surjective homomorphism of algebras.
Then the induced morphism $\Omega _\varphi:\Omega ^\bullet _{A|k}\to\Omega ^\bullet _{B|k}$ is surjective and its kernel is the ideal of $\Omega ^\bullet _{A|k}$ generated by $\Ker \varphi$ and $d(\Ker \varphi)$.
\end{lemma}

\begin{proof}
Set $T:=\Ker \varphi$, then it is enough to consider the case $B=A/T$ and $\varphi :A\to B$ being the canonical projection.
Take the ideal $I$ of $\Omega ^\bullet _{A|k}$ generated by $T$ and $dT$.
$I$ is a dg-ideal, hence, $\Omega ^\bullet _{A|k} /I$ is a dg-algebra and the canonical projection $\pi :\Omega ^\bullet _{A|k} \to \Omega ^\bullet _{A|k}/I$ is a morphism of dg-algebras.
By the universal property there exists a unique morphism of dg-algebras $\mathcal{L}:\Omega ^\bullet _{B|k}\to \Omega ^\bullet _{A|k}/I$ such that $\mathcal{L}|_{B}$ is the identity map.
On the other hand, the kernel of $\Omega _\varphi$ contains $T$ and $dT$.
Therefore, the morphism $\Omega _\varphi$ decomposes as the composition of dg-algebra morphisms $\Omega ^\bullet _{A|k}\xrightarrow{\pi}\Omega ^\bullet _{A|k}/I\xrightarrow{\mathcal{M}} \Omega ^\bullet _{B|k}$.
The composition
$$\Omega ^\bullet _{B|k}\xrightarrow{\mathcal{L}} \Omega ^\bullet _{A|k}/I\xrightarrow{\mathcal{M}} \Omega ^\bullet _{B|k}$$
is clearly the identity map by the universal property of $\Omega ^\bullet _{B|k}$.
Consider the diagram
\begin{equation*}
\begin{tikzcd}
\Omega ^\bullet _{A|k}\arrow[rr,"\pi",bend left =20]\arrow[r,"\Omega _f"']\arrow[rd,"\pi"']&\Omega ^\bullet _{B|k}\arrow[r,"\mathcal{L}"']&\Omega ^\bullet _{A|k}/I\\
&\Omega ^\bullet _{A|k}/I\arrow[u,"\mathcal{M}"']\arrow[ru,"="'].&
\end{tikzcd}
\end{equation*}
The projection map $\pi$ is surjective and the left triangle is clearly commutative.
The upper triangle is commutative by the universal property of $\Omega ^\bullet _{A|k}$.
Hence, the right triangle is commutative.
Therefore, the morphisms $\mathcal{L}$ and $\mathcal{M}$ are mutually inverse.
As the morphism $\pi$ is surjective and its kernel is generated by $T$ and $dT$, we have $\Omega _f$ is surjective and its kernel is generated by $T$ and $dT$.
\end{proof}

\begin{lemma}\label{lem:kahlerformspreservecolimits}
Suppose a $k$-algebra $A$ is the filtered colimit of $k$-algebras $A_i$.
Then $\Omega ^n _{A|k} \cong \colim \Omega ^n _{A_i|k}$.
\end{lemma}
The proof can be found in \cite[Proposition~$4.1$]{KunzKahlerDifferentials}.

We can generalize the notion of algebraic forms to the case of sheaves.
For a topological space $X$ and a sheaf of $k$-algebras $\FC$ on $X$ consider the following presheaf $\Omega ^n _{\FC|k}$ of $\FC$-modules:
$$U\mapsto \Omega ^n _{\FC(U)|k}.$$
For every open $U$ we obtain the complex $\Omega ^\bullet _{\FC(U)|k}$ and, hence, we have the complex of presheaves $\Omega ^\bullet _{\FC |k}$.
The associated sheaves form the complex of sheaves $\OmegaU ^\bullet _{\FC |k}$.

For sets $W^\prime\subset W$ open or closed we fix the notations for the restriction map $\res _{W,W^\prime} :\Omega ^\bullet _{\FC (W)|k} \to \Omega ^\bullet _{\FC (W^\prime)|k}$ and $\res _{W,W^\prime} :\OmegaU ^\bullet _{\FC|k} (W)\to \OmegaU ^\bullet _{\FC|k} (W^\prime)$ induced by the restriction map $\res _{W,W^\prime} :\FC (W)\to \FC (W^\prime)$.

Consider the morphism of complexes of presheaves $\R _X [0]\to \Omega ^\bullet _{\FC |k}$, $1\mapsto 1$, where $\R _X$ is the constant presheaf.
The sheafification functor gives rise to a morphism of complexes of sheaves, the coaugmentation, $\epsilon :\RU _X [0]\to \OmegaU ^\bullet _{\FC |k}$.
\begin{lemma}\label{lem:sheavesomegaaresoft}
For a soft sheaf of algebras $\FC$ on $X$ the sheaves $\OmegaU ^n _{\FC |k}$ are soft.
\end{lemma}
\begin{proof}
The associated sheaf $\OmegaU ^n _{\FC |k}$ has the natural structure of an $\FC$-module.
So it is a sheaf of modules over a soft sheaf and, therefore, is soft by \cite[Theorem~II.$3.7.1$]{godement1958topologie}.
\end{proof}

\section{The map $\Lambda _\FC :\HH ^* (X,\kU _X[0] )\to \HS ^* (\Omega ^\bullet _{\FC (X)|k})$}\label{sec:theoppositemaplambda}
\subsection{The global sections of $\OmegaU ^n _{\FC |k}$}\label{ssec:theglobalsectionsofomegasheaves}
In this subsection we take $\FC$ to be a soft sheaf of $k$-algebras on a compact Hausdorff space $X$.

\begin{lemma}\label{lem:supportsdontintersecthanceproductiszero}
Suppose that $S,F\subset X$ are closed sets such that $S\cap F = \emptyset$.
Then there exists a section $g\in \FC (X)$ such that $\res _{X,S} (g) = 0$ and $\res _{X,F} (g) = 1$.
\end{lemma}
The proof can be found in \cite[Theorem~II.$3.7.2$]{godement1958topologie}.

\begin{lemma}\label{lem:restrictionofformonopensetzero}
Suppose that $U\subset X$ is an open set.
Take $\omega \in \Omega ^n _{\FC (X)|k}$ such that $\omega |_U =0$ in $\Omega ^n _{\FC (U)|k}$.
Take also a section $\varphi\in \FC (X)$ such that $\supp \varphi\subset U$.
Then $\varphi\omega = 0$.
\end{lemma}

\begin{proof}
Put $S:=\supp \varphi$.
The restriction homomorphism $\res _{X,S} :\FC (X)\to \FC (S)$ is surjective as $\FC$ is soft.
By Lemma~\ref{lem:kernelofinducedmapofomegas} the kernel of $\res _{X,S} :\Omega ^\bullet _{\FC (X)|k}\to \Omega ^\bullet _{\FC (S)|k}$ is the ideal generated by $\Ker \res _{X,S}$ and $d(\Ker \res _{X,S})$.
By assumption $\omega \in \Ker \res _{X,S}$ and, hence, it is enough to prove the statement for $\omega$ of the form
$$\omega = dt\wedge\lambda + u\eta ,$$
where $t,u\in\FC (X)$, $\lambda,\eta \in \Omega ^\bullet _{\FC (X)|k}$ and $\res _{X,S} (t) = \res _{X,S} (u) = 0$.
We have
$$\varphi\omega =(\varphi dt)\wedge\lambda + (\varphi u)\eta .$$
We have $\varphi u = 0$ as the supports of $\varphi$ and $u$ do not intersect.
It remains to prove that $\varphi dt = 0$.
By Lemma~\ref{lem:supportsdontintersecthanceproductiszero} there exists a section $g\in \FC(X)$ such that $\res _{X,S} (g) = 0$ and $\res _{X,\supp t}(g) = 1$.
We have
$$\varphi dt = \varphi d(tg)=(\varphi t)dg+(\varphi g)dt=0$$
as $t=tg$, $\varphi t=0$ and $\varphi g=0$.
\end{proof}

\begin{lemma}\label{lem:theglobalsectionsofomegasheaves}
Suppose $\FC$ is a soft sheaf of algebras on $X$.
Then the sheafification $\sh :\Omega ^n _{\FC (X)|k} = \Omega ^n _{\FC|k} (X)\to \OmegaU ^n _{\FC |k} (X)$ is an isomorphism.
\end{lemma}

\begin{proof}
First we prove injectivity.
Take an $n$-form $\omega\in \Omega ^n _{\FC|k} (X)$ such that its image in $\OmegaU ^n _{\FC |k} (X)$ is $0$.
Then there is a finite open cover $(U_i)$ of $X$ such that the restrictions $\omega |_{U_i}$ are $0$ for all $i$.
Choose a partition of unity $(\varphi _i\in \FC (X))$ subordinate to the cover $(U_i)$.
It can always be done by \cite[Theorem~II.$3.6.1$]{godement1958topologie}.
By Lemma~\ref{lem:restrictionofformonopensetzero}, $\varphi _i \omega = 0$ for all $i$.
Now, $\omega = \sum _{i} \varphi _i \omega = 0$.

For surjectivity, take a global section of the sheaf $\OmegaU ^n _{\FC |k}$.
It can be represented by a set of pairs $(U_i, \omega _i)$ where the sets $U_i$ form a finite open cover of $X$ and $\omega _i\in \Omega ^n _{\FC (U_i)|k}$ such that the germs of $\omega _i$ agree at every point.
We seek a form $\omega\in\Omega ^n _{\FC (X)|k}$ such that $\omega _x = (\omega _i)_x$ in $\Omega ^n _{\FC _x|k}$ for each $i$ and $x\in U_i$.
Take a partition of unity $(\varphi _i)$ subordinate to the cover $(U_i)$ and denote $F_i := \supp \varphi _i$.
The restriction map $\res _{X,S}:\Omega ^n _{\FC (X)|k}\to \Omega ^n _{\FC (S)|k}$ is surjective for any closed $S$ by Lemma~\ref{lem:kernelofinducedmapofomegas}.
Hence, we can extend the forms $\res _{X,F_i}(\omega _i)$ to some forms $\bar{\omega} _i\in \Omega ^n _{\FC (X)|k}$. 
Now, put $\omega := \sum _{i} \varphi _i\bar{\omega} _i\in \Omega ^n _{\FC(X)|k}$.
Take $i$ and a point $x\in U_i$.
Introduce $J:=\lbrace j\,|\,x\in F_j \rbrace$.
We have
$$\omega _x = (\sum _{j} \varphi _j \bar{\omega} _j)_x = (\sum _{j\in J} \varphi _j \bar{\omega} _j)_x  = (\sum _{j\in J} \varphi _j \omega _j)_x =(\sum _{j\in J} \varphi _j)_x(\omega _i)_x=(\omega _i)_x.$$
\end{proof}

\subsection{The construction and naturality of $\Lambda _\FC :\HH ^* (X,\kU _X[0] )\to \HS ^* (\Omega ^\bullet _{\FC (X)|k})$}

\begin{definition}\label{def:lambdamap}
For a soft sheaf of algebras $\FC$ on a compact Hausdorff space $X$ we define the map $\Lambda _\FC$ by the following diagram:
\begin{equation*}
\begin{tikzcd}
\HH ^* (X,\kU _X [0] )\arrow[r,"\HH (\epsilon )"]\arrow[ddr,"\Lambda _\FC"',dashed] & \HH ^* (X,\OmegaU ^\bullet _{\FC|k})\\
 & \HS ^* (\OmegaU ^\bullet _{\FC|k} (X)) \arrow[u,"\Upsilon"',"\cong"]\\
 & \HS ^* (\Omega ^\bullet _{\FC (X)|k}) \arrow[u,"\cong","\HS (\sh )"'].
\end{tikzcd}
\end{equation*}
Here the map $\HS (\sh )$ is an isomorphism by Lemma~\ref{lem:theglobalsectionsofomegasheaves} and the map $\Upsilon$ is an isomorphism as the sheaves $\OmegaU ^\bullet _{\FC|k}$ are acyclic by Lemma~\ref{lem:sheavesomegaaresoft}.
\end{definition}

A morphism of ringed spaces $(f,\varphi):(X,\FC )\to (Y,\GC)$ consists of a continuous map $f:X\to Y$ and a morphism of sheaves $\varphi :\GC\to f_*\FC$.

\begin{proposition}\label{prp:lambdaisfunctorial}
The linear map $\Lambda _-$ is natural with respect to morphisms of ringed spaces in the following sense:
for a morphism of ringed spaces $(f,\varphi ):(X,\FC )\to (Y,\GC)$ the following diagram is commutative:
\begin{equation*}
\begin{tikzcd}
\HH ^*(X,\kU _X [0]) \arrow[r,"\Lambda _\FC"] & \HS ^* (\Omega ^\bullet _{\FC (X)|k})\\
\HH ^*(Y,\kU _Y [0]) \arrow[r,"\Lambda _\GC"]\arrow[u,"f^*"] & \HS ^* (\Omega ^\bullet _{\GC (Y)|k})\arrow[u,"\HS (\Omega _{\varphi (Y)})"'].
\end{tikzcd}
\end{equation*}
\end{proposition}

\begin{proof}
The morphism of sheaves $\varphi :\GC\to f_*\FC$ defines the morphism of complexes of sheaves $\OmegaU _{\varphi} :\OmegaU ^\bullet _{\GC |k}\to f_*\OmegaU ^\bullet _{\FC |k}$.

Consider the diagram
\begin{equation*}
\begin{tikzcd}
\HH ^*(X,\kU _X [0])\arrow[r,"\HH (\epsilon )"]\arrow[rrr,"\Lambda _\FC",bend left=15]& \HH ^* (X,\OmegaU ^\bullet _{\FC|k})&\HS ^* (\OmegaU ^\bullet _{\FC|k} (X)) \arrow[l,"\Upsilon"',"\cong"]& \HS ^* (\Omega ^\bullet _{\FC (X)|k})\arrow[l,"\HS (\sh )"',"\cong"]\\
\HH ^*(Y,\kU _Y [0])\arrow[u,"f^*"]\arrow[r,"\HH (\epsilon )"]\arrow[rrr,"\Lambda _\GC",bend right=15]&\HH ^* (Y,\OmegaU ^\bullet _{\GC|k})\arrow[u,"f^*"] &\HS ^* (\OmegaU ^\bullet _{\GC|k} (Y))\arrow[u,"\HS (\OmegaU _\varphi (Y))"]\arrow[l,"\Upsilon"',"\cong"] & \HS ^* (\Omega ^\bullet _{\GC (Y)|k})\arrow[u,"\HS (\Omega _{\varphi (Y)})"]\arrow[l,"\HS (\sh )"',"\cong"].
\end{tikzcd}
\end{equation*}
The middle square is commutative by Lemma~\ref{lem:functorialityofhypercohomology}.
The other squares are commutative for obvious reasons.
\end{proof}

\section{Splitting for the algebra of smooth functions}\label{sec:splittingforthealgebraofsmoothfunctions}
In this section $M$ is a compact smooth manifold.
We prove that for a smooth manifold $M$ the groups $\HH ^* (M, \RU _M [0])$ canonically split off of $\HS ^* (\Omega ^\bullet _{C^\infty (M)|\R})$.

\subsection{Canonical isomorphism $\Theta$}\label{sssec:thetamorphism}
Denote by $\Omega ^\bullet _M$ the complex of sheaves of smooth differential forms on $M$.
We often denote by $\Omega ^\bullet (M)$ the complex $\Omega ^\bullet _M (M)$.
Consider the morphism of complexes, the coaugmentation, $\epsilon :\RU _M [0]\to \Omega ^\bullet _M$ defined by $\epsilon (1):=1$.
We define $\Theta$ by the following commutative diagram
\begin{equation*}
\begin{tikzcd}
\HH ^* (M,\RU _M[0] )\arrow[r,"\HH (\epsilon )","\cong"']\arrow[rr,bend left=15,"\Theta"]&\HH ^* (M,\Omega _{M} ^\bullet )&\HS ^* (\Omega ^\bullet (M))\arrow[l,"\Upsilon"',"\cong"].
\end{tikzcd}
\end{equation*}
The map $\HH (\epsilon )$ in this case is an isomorphism (as $\epsilon$ is a quasi-isomorphism) and $\Upsilon$ is an isomorphism and so $\Theta$ is an isomorphism.

\subsection{Morphisms $\pi$ and $\underline{\pi}$}\label{sssec:pimorphism}
\begin{enumerate}[leftmargin=*]
\item
Take an open set $U\subset M$; the identity map $C^\infty (U)\to \Omega ^0 (U)$ can be uniquely extended to a morphism of dg-algebras $\pi:\Omega ^\bullet _{C^\infty (U)|\R} \to \Omega ^\bullet (U)$ by the universal property of $\Omega ^\bullet _{C^\infty (U)|\R}$.
The following diagram is commutative:
\begin{equation*}
\begin{tikzcd}
C^\infty (U) \arrow [r]\arrow[rd]&\Omega ^\bullet _{C^\infty (U)|\R}\arrow[d,"\pi"]\\
&\Omega ^\bullet (U).
\end{tikzcd}
\end{equation*}
\item
This way we obtain a morphism of complexes of presheaves $\pi :\Omega ^\bullet _{C^\infty _M|\R} \to \Omega ^\bullet _M$.
\item
By the universal property of sheafification we have the morphism of complexes of sheaves $\underline{\pi}:\OmegaU ^\bullet _{C^\infty _M|\R} \to \Omega ^\bullet _M$.
\end{enumerate}

The following diagrams are commutative:
\begin{equation}\label{diag:propertiesofpi}
\begin{tikzcd}
\RU _M [0] \arrow[r,"\epsilon"]\arrow[rd,"\epsilon"']&\OmegaU ^\bullet _{C^\infty _M|\R}\arrow[d,"\underline{\pi}"]&&\Omega ^\bullet _{C^\infty _M|\R}\arrow[r,"\pi"]\arrow[d,"\sh"']&\Omega ^\bullet _M\\
&\Omega^\bullet _M,&&\OmegaU ^\bullet _{C^\infty _M|\R}\arrow[ru,"\underline{\pi}"'].&
\end{tikzcd}
\end{equation}

\subsection{Main splitting theorem}
\begin{theorem}\label{thm:thecompositionforsmoothfunctions}
The following diagram is commutative:
\begin{equation*}
\begin{tikzcd}
\HH ^* (M,\RU _M[0] ) \arrow[r,"\Lambda _{C^\infty _M}"]\arrow[rr,"\Theta",bend left =20]&\HS ^* (\Omega ^\bullet _{C^\infty (M)|\R})\arrow[r,"\HS (\pi )"]&\HS ^* (\Omega ^\bullet (M))
.\end{tikzcd}
\end{equation*}
\end{theorem}
\begin{proof}
Consider the diagram
\begin{equation*}
\begin{tikzcd}
\HH ^* (M,\RU _M [0] )\arrow[rr,"\HH (\epsilon)",bend left = 15]\arrow[rrd,"\Theta" {xshift=-50}]\arrow[r,"\HH (\epsilon )"]\arrow[ddr,"\Lambda _{C^\infty _M}"',bend right =10] & \HH ^* (M,\OmegaU ^\bullet _{C^\infty _M|\R})\arrow[r,"\HH (\underline{\pi} )"]&\HH ^* (M,\Omega ^\bullet _M)\\
 & \HS ^* (\OmegaU ^\bullet _{C^\infty _M|\R} (M)) \arrow[u,"\Upsilon" {yshift =-4},crossing over]\arrow[r,"\HS (\underline{\pi})"']&\HS ^* (\Omega ^\bullet (M))\arrow[u,"\Upsilon"',"\cong"]\\
 & \HS ^* (\Omega ^\bullet _{C^\infty (M)|\R}) \arrow[u,"\HS (\sh )"']\arrow[ru,"\HS (\pi )"',bend right =10].&
\end{tikzcd}
\end{equation*}
The left quadrangle is commutative by the definition of $\Lambda _{C^\infty _M}$.
The upper triangle and the bottom right triangles are commutative by Diagrams~\ref{diag:propertiesofpi}.
The upper right square is commutative by the naturality of $\Upsilon$.
The upper triangle with $\Theta$ is commutative by the definition of $\Theta$.
As the right arrow $\Upsilon$ is an isomorphism, the whole diagram is commutative.
\end{proof}

\section{Simplicial dg-algebra of flat cochains $\Omega ^\bullet _\flat (\Delta ^-)$}\label{sec:simplicialdgalgebraofflatforms}
\subsection{Flat cochains}

We present the dg-algebra $\Omega _\flat ^\bullet (\Delta ^n)$ of flat cochains on the closed $n$-simplex.
We follow \cite[]{whitney2015geometric}, see also \cite{heinonen2005lectures}.
The dg-algebras $\Omega _\flat ^\bullet (\Delta ^n)$, $n\geq 0$, will form the simplicial dg-algebra $\Omega ^\bullet _\flat (\Delta ^-)$.

For a convex set $V\subset \R ^m$ we denote by $\Aff V$ the affine hull of $V$.
We denote by $\relint V$ the interior of $V$ relative to $\Aff V$.

\begin{definition}\label{def:affinechains}
For a convex set $V\subset \R ^m$ denote by $\CS ^{\text{aff}}_k (V) $ the vector space of affine singular $k$-chains in $V$ with coefficients in $\R$.
The boundary map $\partial :\CS ^{\text{aff}}_k (V)\to\CS ^{\text{aff}}_{k-1} (V)$ is defined in the usual manner:
we denote by $\gamma _i :\Delta ^{k-1}\to\Delta ^k$ the $i$-th face embedding, then for a simplex $\sigma :\Delta ^k\to V$ we denote by $\partial _i\sigma : \Delta ^{k-1}\to V$ the composition $\sigma\circ\gamma _i$ and by $\partial \sigma$ the singular chain $\sum _{i=0}^k (-1)^i\partial _i \sigma$.
The vector spaces $\CS ^{\text{aff}}_k (V)$ together with the boundary map form the chain complex $\CS ^{\text{aff}}_\bullet (V)$.
\end{definition}

An affine map $f:V\to V^\prime$ induces the morphism of complexes
$$\CS ^{\text{aff}} (f):\CS ^{\text{aff}}_\bullet (V)\to \CS ^{\text{aff}}_\bullet (V^\prime ).$$
Hence, the correspondence $V\mapsto \CS ^{\text{aff}}_\bullet (V)$ is covariant with respect to affine maps.

Define the \textit{mass} $|\alpha |$ of an affine $k$-chain $\alpha = \sum \lambda _i \sigma _i$, where $\sigma _i$ are distinct singular simplices, as
$$|\alpha | := \sum |\lambda _i| |\sigma _i|_\mathrm{L}$$
where $|\sigma _i|_\mathrm{L}$ denotes the Lebesgue $k$-measure of $\sigma _i$.
Define the flat seminorm $|\cdot |^\flat$ on $\CS ^{\text{aff}}_k (V)$ as
$$|\alpha |^\flat := \inf _{\beta\in\CS ^{\text{aff}}_{k+1} (V)} \lbrace |\alpha - \partial \beta|+|\beta| \rbrace.$$

\begin{lemma}\label{lem:convexsetsinclusionpreservesseminorm}
The map $\CS ^{\text{aff}}_k (V)\to \CS ^{\text{aff}}_k (V^\prime )$ induced by an inclusion of convex sets $V\hookrightarrow V^\prime$ preserves the flat seminorm.
\end{lemma}
\begin{proof}
Follows from \cite[Lemma~V.2b]{whitney2015geometric}.
\end{proof}

\begin{lemma}\label{lem:theimageisdense}
Let $V$ be a convex set.
The map $\CS ^{\text{aff}}_k (\relint V)\to \CS ^{\text{aff}}_k (V)$ induced by the inclusion $\relint V\hookrightarrow V$ has image dense with respect to the flat seminorm.
\end{lemma}
\begin{proof}
Any singular affine simplex in $V$ can be approximated by one in $\relint V$.
Cf. \cite[VIII.1(h)]{whitney2015geometric}.
\end{proof}

\begin{definition}\label{def:flatcochains}
If $V$ is a convex set we define $\Omega _\flat ^k (V)$ as the vector space of linear functionals $\CS ^{\text{aff}}_k (V)\to\R$ bounded with respect to the seminorm $|\cdot |^\flat$.
We call the elements of $\Omega _\flat ^k (V)$ \textit{flat cochains} on $V$.
We define the differential $dX$ of a cochain $X$ by the formula $\left< dX,\alpha \right>:=\left< X,\partial \alpha \right>$.
We obtain the complex $\Omega _\flat ^\bullet (V)$.
This definition is equivalent to the one given by Whitney in \cite[VIII.1(b) and VII.10]{whitney2015geometric}.
\end{definition}

The complex $\CS ^{\text{aff}}_\bullet (V)$ was covariant in $V$ with respect to affine maps, so the complex $\Omega _\flat ^\bullet (V)$ is contravariant in $V$ with respect to affine maps.

We refer the reader to \cite[IX.$14$]{whitney2015geometric} where Whitney defines the graded-commutative multiplication of flat cochains in open sets.
The multiplication is natural with respect to affine maps \cite[X.$11$]{whitney2015geometric}.
For an open convex set $V$ the complex $\Omega _\flat ^\bullet (V)$ becomes a dg-algebra, which is contravariant with respect to affine maps.
Hence, the multiplication is well defined in relatively open convex sets.
Next we wish to define the product of two flat cochains on a closed convex set $V$.
For this we need the following lemma.
\begin{lemma}\label{lem:extendingthecochaintotheclosure}
For a closed convex set $V\subset \R ^m$ the inclusion $\relint V\hookrightarrow V$ induces an isomorphism $\rho :\Omega ^\bullet _\flat (V)\to \Omega ^\bullet _\flat (\relint V)$.
\end{lemma}
\begin{proof}
Follows from Lemmas~\ref{lem:convexsetsinclusionpreservesseminorm} and ~\ref{lem:theimageisdense}.
\end{proof}

For a closed convex set $V\subset\R ^m$ we introduce the multiplication in $\Omega ^\bullet _\flat (V)$ in the way that the isomorphism of complexes $\rho :\Omega ^\bullet _\flat (V) \to \Omega ^\bullet _\flat (\relint V)$ from Lemma~\ref{lem:extendingthecochaintotheclosure} becomes an isomorphism of dg-algebras.
This multiplication and its naturality are implicit in \cite[VII.$12$]{whitney2015geometric}.
In Proposition~\ref{prp:propertiesofmultiplication} we show that the multiplication is natural with respect to affine maps.

Take $V$ a closed convex set.
The inclusion $V\hookrightarrow \Aff V$ induces the map $\pi :\Omega ^\bullet _\flat (\Aff V)\to \Omega ^\bullet _\flat (V)$.

\begin{lemma}\label{lem:piissurjective}
The map $\pi :\Omega ^\bullet _\flat (\Aff V)\to \Omega ^\bullet _\flat (V)$ is surjective.
\end{lemma}

\begin{proof}
Follows from Lemma~\ref{lem:convexsetsinclusionpreservesseminorm} by the Hahn-Banach theorem.
Alternatively, see \cite[VIII.1(h)]{whitney2015geometric} and apply Lemma~\ref{lem:extendingthecochaintotheclosure}.
\end{proof}

\begin{lemma}\label{lem:piisdgmorphism}
The map $\pi :\Omega ^\bullet _\flat (\Aff V)\to \Omega ^\bullet _\flat (V)$ is a morphism of dg-algebras.
\end{lemma}
\begin{proof}
Consider the following diagram
\begin{equation*}
\begin{tikzcd}
\Omega ^\bullet _\flat (\Aff V) \arrow[r,"\pi"] & \Omega ^\bullet _\flat (V) \arrow[r,"\rho"] & \Omega ^\bullet _\flat (\relint V)
.\end{tikzcd}
\end{equation*}
The composition $\rho\circ\pi$ is induced by the inclusion $\relint V\hookrightarrow \Aff V$ and, hence, is a morphism of dg-algebras.
Since $\rho$ is an isomorphism of dg-algebras, $\pi$ is a morphism of dg-algebras.
\end{proof}

\begin{proposition}\label{prp:propertiesofmultiplication}
Consider $f:V\to V^\prime$ an affine map of closed convex sets.
Then the induced morphism $\Omega _\flat (f):\Omega ^\bullet _\flat (V^\prime )\to \Omega ^\bullet _\flat (V)$
preserves multiplication.
\end{proposition}

\begin{proof}
Consider the following commutative diagram:
\begin{equation*}
\begin{tikzcd}[column sep=large]
\Omega ^\bullet _\flat (\Aff V^\prime) \arrow[d,"\pi"']\arrow[r,"\Omega _\flat (\Aff f)"] & \Omega ^\bullet _\flat (\Aff V)\arrow[d,"\pi"]\\
\Omega ^\bullet _\flat (V^\prime )\arrow[r,"\Omega _\flat (f)"] & \Omega ^\bullet _\flat (V)
.\end{tikzcd}
\end{equation*}
Here $\Aff f:\Aff V\to\Aff V^\prime$ is the affine extension of $f$.
The left vertical arrow is surjective by Lemma~\ref{lem:piissurjective}.
The vertical arrows are morphisms of dg-algebras by Lemma~\ref{lem:piisdgmorphism}.
The map $\Omega _\flat (\Aff f)$ is a morphism of dg-algebras as it is induced by an affine map of relatively open sets.
Therefore, $\Omega _\flat (f)$ is a morphism of dg-algebras.
\end{proof}

An order-preserving map $[n]\to [l]$ induces an affine map $\Delta ^n\to \Delta ^l$.
Hence, we obtain the simplicial dg-algebra $\Omega ^\bullet _\flat (\Delta ^-)$ defined as $[n]\mapsto \Omega ^\bullet _\flat (\Delta ^n)$.
In particular, there are face maps $\partial _i :\Omega ^\bullet _\flat (\Delta ^n)\to \Omega ^\bullet _\flat (\Delta ^{k-1})$ induced by the face embeddings $\gamma _i :\Delta ^{k-1}\to \Delta ^n$.

\subsection{The attributes of the simplicial dg-algebra $\Omega ^\bullet _\flat (\Delta ^-)$}\label{ssec:propertiesomegaflat}
\mbox{}
(1)
We define the linear maps $\JC _n:\Omega _\flat ^n (\Delta ^n)\to \R$ as
$$\JC _n(\omega) :=\left< \omega ,\Id _{\Delta ^n} \right>.$$

\begin{proposition}[Stokes' formula]\label{prp:stokesformulaforflatcochains}
For every $\eta\in\Omega _\flat ^{n-1} (\Delta ^{n})$ the following formula holds
$$\JC _n (d\eta) = \sum _{i=0} ^n (-1)^i\JC _{n-1} (\partial _i \eta ).$$
\end{proposition}

\begin{proof}
We have $$\JC _n (d\eta ) = \left< d\eta ,\Id _{\Delta ^n}\right> = \left< \eta ,\partial \,\Id _{\Delta ^n}\right> =
\sum _{i=0} ^n (-1)^i\left< \eta ,\gamma _i\right>=$$ $$=\sum _{i=0} ^n (-1)^i\left< \partial _i\eta ,\Id _{\Delta ^{n-1}} \right> = \sum _{i=0} ^n (-1)^i\JC _{n-1} (\partial _i\eta ).$$
\end{proof}

(2)
The Lipschitz functions on $\Delta ^n$ form the algebra $\Lip (\Delta ^n)$.
The correspondence $[n]\mapsto \Lip (\Delta ^n)$ gives rise to the simplicial algebra $\Lip (\Delta ^-)$.
Every Lipschitz function $f\in \Lip (\Delta ^n)$ defines a flat $0$-cochain $\zeta (f)$ on $\Delta ^n$ in the following way.
Take an affine simplex $\sigma :\lbrace v_0\rbrace =\Delta ^0\to \Delta ^n$ and set $$\left<\zeta (f),\sigma \right> := f(\sigma (v_0)).$$
It is easy to check that $0$-cochain $\zeta (f)$ is flat (see \cite[Theorem~VII.4B]{whitney2015geometric}.
We obtain the linear map $\zeta :\Lip (\Delta ^n)\to \Omega ^\bullet _\flat (\Delta ^n)$ with image in $\Omega ^0 _\flat (\Delta ^n)$.
The composition $\rho\circ\zeta :\Lip (\Delta ^n)\to \Omega ^\bullet _\flat (\Delta ^n)\xrightarrow{\cong}\Omega ^\bullet _\flat (\relint\Delta ^n)$ is a homomorphism of algebras (by the definition of multiplication on open sets \cite[IX.$14$]{whitney2015geometric}).
Hence, $\zeta$ is a homomorphism of algebras.
The map $\zeta$ clearly preserves the simplicial structure and we obtain the morphism of simplicial algebras
$$\zeta :\Lip (\Delta ^-)\to \Omega ^0 _\flat (\Delta ^-).$$

(3)
Consider the simplicial dg-algebra $\Omega ^\bullet (\Delta ^-)$ of smooth differential forms.
Every smooth form $\theta\in\Omega ^k (\Delta ^n)$ gives rise to a flat cochain $\nabla (\theta )\in \Omega ^k _\flat (\Delta ^n)$ by
$$\left< \nabla (\theta ),\sigma \right>:=\int _{\Delta ^k} \sigma ^*\theta$$
for affine $\sigma :\Delta ^k \to \Delta ^n$.
The flatness can be easily checked (see \cite[V.14 and Theorem~V.10A]{whitney2015geometric}).
This correspondence gives rise to a simplicial linear map
$$\nabla :\Omega ^\bullet (\Delta ^-)\to \Omega ^\bullet _\flat (\Delta ^-).$$
By Stokes' formula for smooth forms the differential is preserved.
The composition $$\Omega ^\bullet  (\Delta ^n)\xrightarrow{\nabla}  \Omega ^\bullet _\flat (\Delta ^n)\xrightarrow{\rho} \Omega ^\bullet _\flat (\relint \Delta ^n)$$ preserves multiplication by the definition of multiplication of flat cochains \cite[X.$14$]{whitney2015geometric}.
Hence, by the definition of multiplication in $\Omega ^\bullet _\flat (\Delta ^n)$ the map $\nabla$ preserves multiplication.
Therefore, the map $$\nabla :\Omega ^\bullet (\Delta ^-)\to \Omega ^\bullet _\flat (\Delta ^-)$$ is a morphism of simplicial dg-algebras.

(4)
We define the linear maps $\IC _n :\Omega ^n(\Delta ^n)\to \R$ as
$$\IC _n (\omega ):= \int _{\Delta ^n} \omega.$$
The diagram
\begin{equation*}
\begin{tikzcd}
\Omega ^n(\Delta ^n) \arrow[r,"\IC _n"]\arrow[d,"\nabla"] & \R\\
\Omega ^n _\flat (\Delta ^n)\arrow[ur,"\JC _n",swap]&
\end{tikzcd}
\end{equation*}
is commutative by construction.

(5)
The smooth functions form the simplicial algebra $C^\infty (\Delta ^-)$.
It embeds via $i:C^\infty (\Delta ^-)\hookrightarrow \Lip (\Delta ^-)$ in the simplicial algebra $\Lip (\Delta ^-)$.
The following diagram of morphisms of simplicial algebras is commutative:
\begin{equation*}
\begin{tikzcd}
C^\infty (\Delta ^-)\arrow[r,hook]\arrow[d,hook,"i"'] & \Omega ^\bullet (\Delta ^-)\arrow[d,"\nabla"]\\
\Lip (\Delta ^-)\arrow[r,"\zeta "]&\Omega ^\bullet _\flat(\Delta ^-)
.\end{tikzcd}
\end{equation*}

\section{The map $\Psi _A :\HS ^* (\Omega ^\bullet _{A|\R})\to \HH ^* (X,\RU _{X} [0])$}\label{sec:theconstructionofpsi}
\subsection{The pullback of an algebraic de Rham form}\label{ssec:thepullbackofanalgebraicderhamform}
Take $B$ a finitely generated $\R$-algebra.
We construct a morphism of dg-algebras $\mu (\sigma ):\Omega ^\bullet _{B|\R} \to \Omega _\flat ^\bullet (\Delta ^n)$ for every singular Lipschitz simplex $\sigma :\Delta ^n\to \spec B$.
For such $\sigma$ we define an algebra homomorphism $\theta (\sigma ):B\to \Lip (\Delta ^n)$ as $b\mapsto \hat{b} \circ\sigma$.
By the universal property of $\Omega ^\bullet _{B|\R}$ there exists a unique morphism of dg-algebras $\mu (\sigma ):\Omega _{B|\R} ^\bullet \to \Omega _\flat ^\bullet (\Delta ^n)$ making the following diagram commute:
\begin{equation*}
\begin{tikzcd}
\Omega ^\bullet _{B|\R} \arrow[r,"\mu (\sigma )",dashed] & \Omega _\flat ^\bullet (\Delta ^n)\\
B\arrow[u]\arrow[r,"\theta (\sigma )"]& \Lip (\Delta ^n)\arrow[u,"\zeta"]
.\end{tikzcd}
\end{equation*}

\begin{lemma}\label{lem:simplicialmu}
The following diagram is commutative:
\begin{equation*}
\begin{tikzcd}
\Omega ^\bullet _{B|\R} \arrow[d,"\mu (\sigma)"']\arrow[rd,"\mu (\partial _i\sigma)"] &\\
\Omega ^\bullet _\flat (\Delta ^n)\arrow[r,"\partial _i"] & \Omega ^\bullet _\flat (\Delta ^{n-1})
.\end{tikzcd}
\end{equation*}
\end{lemma}
\begin{proof}
Consider the following diagram:

\begin{equation*}
\begin{tikzcd}
\Omega ^\bullet _{B|\R}\arrow[r,"\mu (\sigma)"']\arrow[rr,"\mu (\partial _i\sigma)",bend left =20] & \Omega ^\bullet _\flat (\Delta ^n)\arrow[r,"\partial _i"'] & \Omega _\flat ^\bullet (\Delta ^{n-1})\\
B \arrow[u]\arrow[r,"\theta (\sigma)"]\arrow[rr,"\theta(\partial _i \sigma)"',bend right =20] & \Lip (\Delta ^n)\arrow[u,"\zeta"]\arrow[r,"\partial _i"] & \Lip (\Delta ^{n-1})\arrow[u,"\zeta"']
.\end{tikzcd}
\end{equation*}
The right square obviously commutes.
By the definition of $\mu$ the left square and the outer contour commute.
The bottom triangle commutes by the definition of $\theta$.
Hence, by the universal property of $\Omega ^\bullet _{B|\R}$ the upper triangle also commutes.
\end{proof}

The map $\mu (\sigma )$ is natural in algebra, namely:
\begin{lemma}\label{lem:pullbackofanalgebraicformisfunctorial}
Suppose $\varphi:B^\prime\to B$ is a homomorphism of finitely generated $\R$-algebras.
Consider a commutative diagram
\begin{equation*}
\begin{tikzcd}
\Delta ^n \arrow[r,"\sigma"]\arrow[rd,"\sigma ^\prime"'] & \spec B \arrow[d, "\spec \varphi"]\\
& \spec B^\prime
\end{tikzcd}
\end{equation*}
with $\sigma$ and $\sigma ^\prime$ Lipschitz.
Then the following diagram commutes:
\begin{equation*}
\begin{tikzcd}
\Omega ^\bullet _{B|\R} \arrow[r,"\mu (\sigma)"] & \Omega _\flat ^\bullet (\Delta ^n)\\
\Omega ^\bullet _{B^\prime|\R}\arrow[u,"\Omega _\varphi"] \arrow[ru,"\mu (\sigma ^\prime)"']. & 
\end{tikzcd}
\end{equation*}
\end{lemma}
\begin{proof}
Consider the following diagram:
\begin{equation*}
\begin{tikzcd}
\Omega ^\bullet _{B^\prime|\R}\arrow[r,"\Omega _\varphi"']\arrow[rr,"\mu (\sigma ^\prime)",bend left =20] & \Omega ^\bullet _{B|\R}\arrow[r,"\mu (\sigma )"'] & \Omega _\flat ^\bullet (\Delta ^n)\\
B^\prime \arrow[u]\arrow[r,"\varphi"]\arrow[rr,"\theta (\sigma ^\prime)"',bend right =20] & B\arrow[u]\arrow[r,"\theta (\sigma )"] & \Lip (\Delta ^n)\arrow[u,"\zeta"']
.\end{tikzcd}
\end{equation*}
The left square obviously commutes.
By the definition of $\mu$ the right square and the outer contour commute.
The bottom triangle commutes: for $c\in B^\prime$
$$\theta (\sigma)(\varphi(c))=\widehat{\varphi(c)}\circ\sigma=\hat{c}\circ (\spec \varphi)\circ \sigma = \hat{c}\circ \sigma ^\prime = \theta(\sigma ^\prime)(c).$$
Hence, by the universal property of $\Omega ^\bullet _{B^\prime|\R}$ the upper triangle also commutes.
\end{proof}

\subsection{The map $\xi _B:\Omega ^\bullet _{B|\R}\to \CS ^\bullet _{\Lip} (\spec B)$}\label{ssec:comparisonmapforfinitelygeneratedsubalgebra}
Let $B$ be a finitely generated $\R$-algebra.
To an algebraic $n$-form $\omega\in\Omega^n _{B|\R}$ we associate a Lipschitz singular cochain $\xi _B (\omega )$.
On a Lipschitz singular simplex $\sigma :\Delta ^n \to \spec B$ we define $\xi _B (\omega )$ as
\begin{equation*}
\left< \xi _B (\omega ),\sigma \right> := \JC _n (\mu (\sigma ) (\omega ))\in\R.
\end{equation*}

\begin{proposition}\label{prp:morphismfromomega}
The above gives a morphism of complexes
$$\xi _B:\Omega^\bullet _{B|\R} \to \CS ^\bullet _{\Lip} (\spec B).$$
\end{proposition}
\begin{proof}
The only part that needs to be checked is that the map preserves $d$.
Take a Lipschitz singular simplex $\sigma :\Delta ^n\to \spec B$ and $\eta\in\Omega ^{n-1} _{B|\R}$, then
$$\left< \xi _B (d\eta ),\sigma \right>=\JC _n (\mu (\sigma ) (d\eta )) \overset{(2)}{=} \JC _n (d(\mu (\sigma ) (\eta ))) \overset{(3)}{=} \sum _{i=0} ^n (-1)^i \JC _{n-1} (\partial _i (\mu (\sigma ) (\eta )))\overset{(4)}{=}$$
$$\overset{(4)}{=}\sum _{i=0} ^n (-1)^i \JC _{n-1} (\mu (\partial _i \sigma ) (\eta ))= \sum _{i=0} ^n (-1)^i \left< \xi _B (\eta ),\partial _i\sigma \right> = \left< \xi _B (\omega ),\partial\sigma \right> = \left< d\xi _B (\omega ),\sigma \right>.$$
The second equality follows from the fact that $\mu (\sigma )$ is a morphism of complexes.
The third equality follows from Proposition~\ref{prp:stokesformulaforflatcochains}.
The fourth equality follows from Lemma~\ref{lem:simplicialmu}.
\end{proof}

Lemma~\ref{lem:pullbackofanalgebraicformisfunctorial} allows us to prove the naturality of $\xi$ in algebras:
\begin{lemma}\label{lem:functorialityinalgebrapsi}
Suppose $\varphi :B^\prime\to B$ is a homomorphism of finitely generated $\R$-algebras.
Then the following diagram is commutative
\begin{equation*}
\begin{tikzcd}[column sep = large]
\Omega ^\bullet _{B|\R} \arrow[r,"\xi _B"]&  \CS ^\bullet _{\Lip} (\spec B)\\
\Omega ^\bullet _{B^\prime|\R} \arrow[u,"\Omega _\varphi "]\arrow[r,"\xi _{B^\prime}"] & \CS ^\bullet _{\Lip} (\spec B^\prime)\arrow[u,"\CS _{\Lip} (\spec \varphi )"']
.\end{tikzcd}
\end{equation*}
\end{lemma}

\subsection{The map $\Phi _B:\HS ^* (\Omega ^\bullet _{B|\R}) \to\HH ^* (\spec B,\RU _{\spec B}[0])$}\label{ssec:thecomparisonmap}
\mbox{}

We define the homomorphism $\Phi _B$ as the vertical map making the following diagram commutative:
\begin{equation*}
\begin{tikzcd}
\HH ^* (\spec B,\RU _{\spec B} [0]) \arrow[r, "\HH (\epsilon )","\cong"'] & \HH ^* ({\spec B}, \CU ^\bullet _{\Lip ,{\spec B}})\\
\HS ^* (\Omega ^\bullet _{B|\R})\arrow[r,"\HS (\xi _B)"]\arrow[u,"\Phi _B",dashed] & \HS ^* (\CS ^\bullet _{\Lip} ({\spec B}))\arrow[u,"\cong","\Upsilon\circ\HS (\sh )"']
.\end{tikzcd}
\end{equation*}
Here $\HH (\epsilon )$ is an isomorphism by Corollary~\ref{cor:lipschitzcochainsresolution}.

\begin{lemma}\label{lem:phiisfunctorialinalgebra}
Let $\varphi :B^\prime\to B$ be a homomorphism of finitely generated $\R$-algebras.
Then the following diagram is commutative:
\begin{equation*}
\begin{tikzcd}
\HS ^* (\Omega ^\bullet _{B|\R})\arrow[r,"\Phi _{B}"] & \HH ^* (\spec B,\RU _{\spec B} [0]) \\
\HS ^* (\Omega ^\bullet _{B^\prime|\R})\arrow[u,"\HS (\Omega _\varphi )"]\arrow[r,"\Phi _{B^\prime}"'] & \HH ^* (\spec B^\prime,\RU _{\spec B^\prime} [0])\arrow[u, "(\spec \varphi )^*"']
.\end{tikzcd}
\end{equation*}
\end{lemma}
\begin{proof}
It follows directly from Lemma~\ref{lem:functorialityinalgebrapsi} and naturality of $\Upsilon\circ\HS (\sh)$.
\end{proof}

\begin{remark}
There are other ways to define a map, analogous to $\Phi _B$.
One way is to consider semi-algebraic cochains, instead of Lipschitz ones, see \cite[]{hardt2011real}.
Another way would be to use the \textit{filtered de Rham complex}, see \cite[Proposition~$7.24$]{peters2008mixed}.
Our construction of $\Phi _B$ allows us to relate it to the classical de Rham complex $\Omega ^\bullet (M)$ in the case $B\subset C^\infty (M)$, see Lemma~\ref{lem:phiandpidiagram}.
\end{remark}

\subsection{The map $\Psi _A :\HS ^* (\Omega ^\bullet _{A|\R})\to \HH ^* (X,\RU _X [0])$}

For a topological space $X$ and a subalgebra $A\subset C(X)$ write $A$ as the filtered colimit of its finitely generated subalgebras: $A=\colim B$.
The functors $\Omega$ and $\HS$ preserve filtered colimits.
For an inclusion of finitely generated subalgebras $i:B^\prime\hookrightarrow B$ of $C(X)$ consider the diagram
\begin{equation*}
\begin{tikzcd}
\HS ^* (\Omega ^\bullet _{B|\R}) \arrow[r,"\Phi _B"]&\HH ^* (\spec B,\RU _{\spec B} [0])\arrow[r,"\Gamma ^* _B"]&\HH ^* (X,\RU _{X} [0])\\
\HS ^* (\Omega ^\bullet _{B^\prime|\R})\arrow[u,"\HS (\Omega _i)"] \arrow[r,"\Phi _{B^\prime}"]&\HH ^* (\spec B^\prime,\RU _{\spec B^\prime} [0])\arrow[u,"(\spec i)^*"]\arrow[ru,"\Gamma ^* _{B^\prime}"'].&
\end{tikzcd}
\end{equation*}
The left square is commutative by Lemma~\ref{lem:phiisfunctorialinalgebra} and the right triangle is commutative by the naturality of $\Gamma _B$ (see Subsection~\ref{ssec:algebraicsets}).
We pass to the colimit of $\Gamma ^* _B \circ \Phi _B$ over all finitely generated subalgebras $B\subset A$ and obtain the map
$$\Psi _A :\HS ^* (\Omega ^\bullet _{A|\R}) \to \HH ^* (X,\RU _{X} [0]).$$

\begin{proposition}\label{prp:functorialityofpsiinspaces}
Take $f:X\to Y$ a continuous map of topological spaces.
Suppose that $A\subset C(X)$ and $A^\prime\subset C(Y)$ are subalgebras and $\varphi :A^\prime\to A$ is a homomorphism such that the following diagram is commutative:
\begin{equation*}
\begin{tikzcd}
C(X)&A\arrow[l,hook']\\
C(Y)\arrow[u,"f^*"]&A^\prime \arrow[l,hook']\arrow[u,"\varphi"'].
\end{tikzcd}
\end{equation*}
Then the following diagram is commutative:
\begin{equation*}
\begin{tikzcd}
\HS ^* (\Omega ^\bullet _{A|\R}) \arrow[r,"\Psi _{A}"] & \HH ^* (X,\RU _{X} [0])\\
\HS ^* (\Omega ^\bullet _{A^\prime|\R}) \arrow[u,"\HS(\Omega _{\varphi})"]\arrow[r,"\Psi _{A^\prime}"] & \HH ^* (Y,\RU _{Y} [0])\arrow[u,"f^*"']
.\end{tikzcd}
\end{equation*}
\end{proposition}
\begin{proof}
It follows directly from Lemma~\ref{lem:phiisfunctorialinalgebra} and naturality of $\Gamma$.
\end{proof}

\section{Identifying the map $\Psi _{C^\infty (M)}$}\label{sec:theclassicalderhamcomparisonmap}

\subsection{The map $\xi ^\sm _B:\Omega ^\bullet _{B|\R}\to \CS ^\bullet _{\sm} (\spec B)$}\label{ssec:comparisonmapforfinitelygeneratedsubalgebrasmooth}
This subsection mirrors the Subsections~\ref{ssec:thepullbackofanalgebraicderhamform} and ~\ref{ssec:comparisonmapforfinitelygeneratedsubalgebra}, so the proofs will be omitted.

Take $B$ a finitely generated $\R$-algebra.
We construct a morphism of dg-algebras $\mu ^\sm (\sigma ):\Omega ^\bullet _{B|\R} \to \Omega ^\bullet (\Delta ^n)$ for every smooth singular simplex $\sigma :\Delta ^n\to \spec B$.
For such $\sigma$ we define an algebra homomorphism $\theta ^\sm (\sigma ):B\to C^\infty (\Delta ^n)$ as $b\mapsto \hat{b} \circ\sigma$.
By the universal property of $\Omega ^\bullet _{B|\R}$ there exists a unique morphism of dg-algebras $\mu ^\sm (\sigma ):\Omega _{B|\R} ^\bullet \to \Omega ^\bullet (\Delta ^n)$ making the following diagram commute:
\begin{equation*}
\begin{tikzcd}
\Omega ^\bullet _{B|\R} \arrow[r,"\mu ^\sm (\sigma )",dashed] & \Omega ^\bullet (\Delta ^n)\\
B\arrow[u]\arrow[r,"\theta ^\sm (\sigma )"]&  C^\infty (\Delta ^n)\arrow[u,hook]
.\end{tikzcd}
\end{equation*}

\begin{lemma}[cf. Lemma~\ref{lem:simplicialmu}]\label{lem:simplicialmusm}
The following diagram is commutative:
\begin{equation*}
\begin{tikzcd}
\Omega ^\bullet _{B|\R} \arrow[d,"\mu ^\sm (\sigma)"']\arrow[rd,"\mu ^\sm (\partial _i\sigma)"] &\\
\Omega ^\bullet (\Delta ^n)\arrow[r,"\partial _i"] & \Omega ^\bullet (\Delta ^{n-1})
.\end{tikzcd}
\end{equation*}
\end{lemma}

The map $\mu ^\sm (\sigma )$ is natural in algebra, namely:
\begin{lemma}[cf. Lemma~\ref{lem:pullbackofanalgebraicformisfunctorial}]\label{lem:pullbackofanalgebraicformisfunctorialsmooth}
Suppose $\varphi:B^\prime\to B$ is a homomorphism of finitely generated $\R$-algebras.
Consider a commutative diagram
\begin{equation*}
\begin{tikzcd}
\Delta ^n \arrow[r,"\sigma"]\arrow[rd,"\sigma ^\prime"'] & \spec B \arrow[d, "\spec \varphi"]\\
& \spec B^\prime
\end{tikzcd}
\end{equation*}
with $\sigma$ and $\sigma ^\prime$ smooth.
Then the following diagram commutes:
\begin{equation*}
\begin{tikzcd}
\Omega ^\bullet _{B|\R} \arrow[r,"\mu ^\sm (\sigma)"] & \Omega _\flat ^\bullet (\Delta ^n)\\
\Omega ^\bullet _{B^\prime|\R}\arrow[u,"\Omega _\varphi"] \arrow[ru,"\mu ^\sm(\sigma ^\prime)"']. & 
\end{tikzcd}
\end{equation*}

\end{lemma}

Next we construct the map $\xi ^\sm _B:\Omega ^\bullet _{B|\R}\to \CS ^\bullet _{\sm} (\spec B)$ for a finitely generated $\R$-algebra $B$.
For an algebraic $n$-form $\omega\in\Omega ^n _{B|\R}$ and a smooth simplex $\sigma :\Delta ^n \to \spec B$ we set
\begin{equation*}
\left< \xi ^\sm _B (\omega ),\sigma \right> := \IC _n (\mu ^\sm (\sigma ) (\omega ))\in\R
\end{equation*}
(the map $\IC _n$ was defined in Paragraph~\ref{ssec:propertiesomegaflat}(4)).

\begin{proposition}[cf. Proposition~\ref{prp:morphismfromomega}]\label{prp:morphismfromomegasmooth}
The above gives a morphism of complexes
$$\xi ^\sm _B:\Omega^\bullet _{B|\R} \to \CS ^\bullet _{\sm} (\spec B).$$
\end{proposition}

Lemma~\ref{lem:pullbackofanalgebraicformisfunctorialsmooth} allows us to prove the naturality of $\xi ^\sm$ in algebras:
\begin{lemma}[cf. Lemma~\ref{lem:functorialityinalgebrapsi}]\label{lem:functorialityinalgebrapsismooth}

Suppose $\varphi :B^\prime\to B$ is a homomorphism of finitely generated $\R$-algebras.
Then the following diagram is commutative:
\begin{equation*}
\begin{tikzcd}[column sep = large]
\Omega ^\bullet _{B|\R} \arrow[r,"\xi ^\sm _B"]&  \CS ^\bullet _{\sm} (\spec B)\\
\Omega ^\bullet _{B^\prime|\R} \arrow[u,"\Omega _\varphi "]\arrow[r,"\xi ^\sm _{B^\prime}"] & \CS ^\bullet _{\sm} (\spec B^\prime)\arrow[u,"\CS _{\sm} (\spec \varphi )"']
.\end{tikzcd}
\end{equation*}
\end{lemma}

\subsection{Comparing $\xi _B$ and $\xi ^\sm _B$}

Take $B$ a finitely generated $\R$-algebra.
\begin{lemma}\label{lem:smoothlipnablaconnection}
Take $\sigma :\Delta ^n\to \spec B$ a smooth simplex.
Then the following diagram is commutative:
\begin{equation*}
\begin{tikzcd}
\Omega ^\bullet _{B|\R} \arrow[rd,"\mu (\sigma )"']\arrow[r,"\mu ^{\sm} (\sigma )"] &\Omega ^\bullet (\Delta ^n) \arrow[d,"\nabla"]\\
 & \Omega _\flat ^\bullet (\Delta ^n).
\end{tikzcd}
\end{equation*}
\end{lemma}

\begin{proof}
Consider the following diagram:
\begin{equation*}
\begin{tikzcd}
\Omega ^\bullet _{B|\R}\arrow[r,"\mu ^{\sm} (\sigma )"']\arrow[rr,"\mu (\sigma)",bend left =20] & \Omega ^\bullet (\Delta ^n)\arrow[r,"\nabla"'] & \Omega _\flat ^\bullet (\Delta ^n)\\
B\arrow[u]\arrow[r,"\theta ^\sm (\sigma )"]\arrow[rr,"\theta (\sigma )"',bend right = 20] & C^\infty (\Delta ^n)\arrow[u]\arrow[r,hook] & \Lip (\Delta ^n)\arrow[u,"\zeta"']
.\end{tikzcd}
\end{equation*}
The left square and the outer contour commute by the definitions of $\mu$ and $\mu ^\sm$.
The right square commutes by Paragraph~\ref{ssec:propertiesomegaflat}(5).
The bottom triangle commutes by the definitions of $\theta$ and $\theta ^\sm$.
Hence, by the universal property of $\Omega ^\bullet _{B|\R}$ the upper triangle also commutes.
\end{proof}

\begin{lemma}
The following diagram is commutative:
\begin{equation}\label{diag:lipandsmooth}
\begin{tikzcd}
\Omega ^\bullet _{B|\R} \arrow[rd,"\xi _B"']\arrow[r,"\xi ^\sm _B"]&\CS ^\bullet _{\sm}(\spec B)\\
&\CS ^\bullet _\Lip(\spec B)\arrow[u,"\res"']. 
\end{tikzcd}
\end{equation}
Here $\res$ is the restriction of Lipschitz cochains to smooth chains.
\end{lemma}
\begin{proof}

Consider the diagram:
\begin{equation}
\begin{tikzcd}
\Omega ^\bullet _{B|\R} \arrow[r,"\mu ^{\sm}(\sigma )"]\arrow[rd,"\mu  (\sigma )",swap]& \Omega ^\bullet (\Delta ^n)\arrow[d,"\nabla",swap]\arrow[r,"\IC _n"]&\R\\
&\Omega ^\bullet _\flat (\Delta ^n)\arrow[ur,"\JC _n",swap].& 
\end{tikzcd}
\end{equation}
The left triangle is commutative by Lemma~\ref{lem:smoothlipnablaconnection}
The right triangle is commutative by Paragraph~\ref{ssec:propertiesomegaflat}(4).
Hence, the statement follows by the definitions of $\xi$ and $\xi ^\sm$. 
\end{proof}

\subsection{Morphisms $\varkappa$ and $\underline{\varkappa}$}\label{sssec:kappamorphism}
Let $M$ be a smooth manifold.
Take an open set $U\subset M$, a smooth simplex $\sigma :\Delta ^n\to U$ and a smooth $n$-form $\omega\in\Omega ^n (U)$.
We define $\left< \varkappa (\omega),\sigma \right>:=\IC _n( \sigma ^* (\omega ) )$.
By Stokes' formula we obtain the morphism of complexes $$\varkappa :\Omega ^\bullet (U)\to \CS ^\bullet _{\sm} (U).$$
This way we obtain a morphism of complexes of presheaves $\varkappa :\Omega ^\bullet _{M} \to \CS ^\bullet _{\sm ,M}$.
Taking the composition with the sheafification map we get the morphism of complexes of sheaves
$$\underline{\varkappa} :\Omega ^\bullet _{M} \to \CU ^\bullet _{\sm ,M}.$$

The following diagram is commutative:
\begin{equation}\label{diag:propertiesofkappa}
\begin{tikzcd}
\RU _M [0] \arrow[r,"\epsilon"]\arrow[rd,"\epsilon"']&\Omega^\bullet _M\arrow[d,"\underline{\varkappa}"]\\
&\CU ^\bullet _{\sm ,M}.
\end{tikzcd}
\end{equation}
The following diagram is also commutative:
\begin{equation}\label{diag:kappacommutative}
\begin{tikzcd}
\HS ^* (\CS ^\bullet _{\sm} (M))\arrow[r,"\HS (\sh )"]&\HS ^* (\CU ^\bullet _{\sm} (M))\arrow[r,"\Upsilon"]&\HH ^* (M,\CU ^\bullet _{\sm ,M})\\
&\HS ^* (\Omega ^\bullet (M))\arrow[lu,"\HS (\varkappa )"]\arrow[u,"\HS (\underline{\varkappa}(M) )"']\arrow[r,"\Upsilon"]&\HH ^* (M,\Omega ^\bullet _M)\arrow[u,"\HH(\underline{\varkappa})"']
\end{tikzcd}
\end{equation}

\subsection{Identifying the map $\Psi _B$}\label{ssec:auxiliarylemmas}
In this subsection we take $M$ a smooth manifold.

\begin{lemma}\label{lem:identmapfirstlemma}
For the inclusion of a finitely generated subalgebra $i:B\hookrightarrow C^\infty (M)$ and a smooth singular simplex $\sigma :\Delta ^n\to M$ the following diagram commutes:
\begin{equation*}
\begin{tikzcd}
\Omega ^\bullet _{B|\R}\arrow[rrd,"\mu ^{\sm} (\Gamma _B \circ\sigma )"]\arrow[d,"\Omega _i"']&&\\
\Omega ^\bullet _{C^\infty (M)|\R}\arrow[r,"\pi"']&\Omega ^\bullet (M)\arrow[r,"\sigma ^*"']&\Omega ^\bullet (\Delta ^n)
.\end{tikzcd}
\end{equation*}
\end{lemma}

\begin{proof}
Consider the diagram:
\begin{equation*}
\begin{tikzcd}
\Omega ^\bullet _{B|\R}\arrow[rrr,"\mu ^{\sm} (\Gamma _B \circ\sigma )",bend left = 20]\arrow[r,"\Omega _i"']&\Omega ^\bullet _{C^\infty (M)|\R}\arrow[r,"\pi"']&\Omega ^\bullet (M)\arrow[r,"\sigma ^*"']&\Omega ^\bullet (\Delta ^n)\\
B\arrow[u]\arrow[rrr,"\theta ^{\sm} (\Gamma _B \circ\sigma )"',bend right = 20]\arrow[r,hook,"i"]&C^\infty (M)\arrow[r,equal]\arrow[u]&C^\infty (M)\arrow[r,"\sigma ^*"]\arrow[u]&C^\infty (\Delta ^n)\arrow[u]
.\end{tikzcd}
\end{equation*}

The left and right squares clearly commute.
The middle square commutes by the definition of $\pi$.
The outer contour commutes by the definition of $\mu ^\sm$.
For $b\in B$ we have
$$\theta ^{\sm} (\Gamma _B \circ\sigma ) (b) = \hat{b} \circ \Gamma _B \circ\sigma = b\circ \sigma = \sigma ^* (b)$$
by the definition of $\theta ^\sm$ and Equation~\ref{eq:hatgamma}.
Hence, the bottom quadrangle commutes.
By the universal property of $\Omega ^\bullet _{B|\R}$ the upper quadrangle commutes.
\end{proof}

\begin{lemma}\label{lem:integrationcommutativity}
For the inclusion of a finitely generated subalgebra $i:B\hookrightarrow C^\infty (M)$ the following diagram is commutative:
\begin{equation*}
\begin{tikzcd}
\Omega ^\bullet _{B|\R}\arrow[rr,"\xi ^\sm _B"]\arrow[d,"\Omega _i"']&&\CS ^\bullet _\sm (\spec B)\arrow[d,"\CS _\sm (\Gamma _B)"]&\\
\Omega ^\bullet _{C^\infty (M)|\R}\arrow[r,"\pi"]&\Omega ^\bullet (M)\arrow[r,"\varkappa"]&\CS ^\bullet _\sm (M)
\end{tikzcd}
\end{equation*}
\end{lemma}

\begin{proof}
For a form $\omega\in\Omega ^\bullet _{B|\R}$ we have
$$\left< \CS _\sm (\Gamma _B)( \xi ^\sm _B(\omega )),\sigma \right> =\left< \xi ^\sm _B(\omega),\Gamma _B \circ\sigma \right>= \IC _n (\mu ^{\sm} (\Gamma _B \circ\sigma ) (\omega ))\overset{(3)}{=}$$
$$\overset{(3)}{=}\IC _n (\sigma ^*(\pi (\Omega _i (\omega )))) =\left< \varkappa (\pi (\Omega _i (\omega ))),\sigma\right>.$$
The equality (3) follows from Lemma~\ref{lem:identmapfirstlemma}.
\end{proof}

\begin{lemma}\label{lem:phiandpidiagram}
For the inclusion of a finitely generated subalgebra $i:B\hookrightarrow C^\infty (M)$ the following diagram is commutative:
\begin{equation*}\label{diag:psicommutativelemma}
\begin{tikzcd}
\HS ^*(\Omega ^\bullet _{B|\R})\arrow[rr,"\Phi _B"]\arrow[d,"\HS (\Omega _i)"']&&\HH ^* (\spec B,\RU _{\spec B} [0])\arrow[d,"\Gamma _B ^*"]\\
\HS ^*(\Omega ^\bullet _{C^\infty (M)|\R})\arrow[r,"\HS (\pi)"]&\HS ^*(\Omega ^\bullet (M))&\HH ^* (M,\RU _{M} [0])\arrow[l,"\Theta"].
\end{tikzcd}
\end{equation*}
\end{lemma}
\begin{proof}

Consider the diagram:
\begin{equation}\label{diag:rightpentagon}
\begin{tikzcd}
\HS ^* (\CS ^\bullet _{\Lip} (\spec B))\arrow[r,"\Upsilon\circ\HS (\sh )","\cong"']\arrow[d,"\HS (\res )"']&\HH ^* (\spec B,\CU ^\bullet _{\Lip ,\spec B})\arrow[d,"\HH (\res )"']&\\
\HS ^* (\CS ^\bullet _{\sm} (\spec B))\arrow[d,"\HS (\CS _\sm (\Gamma _B))"']\arrow[r,"\Upsilon\circ\HS (\sh )","\cong"']&\HH ^* (\spec B,\CU ^\bullet _{\sm ,\spec B})\arrow[d,"\Gamma _B ^*"']& \HH ^*(\spec B ,\RU _{\spec B}[0])\arrow[lu,"\HH (\epsilon )"',"\cong",bend right = 15]\arrow[l,"\HH (\epsilon )"']\arrow[d,"\Gamma _B ^*"]\\
\HS ^* (\CS ^\bullet _{\sm } (M))\arrow[r,"\Upsilon\circ\HS (\sh )","\cong"']&\HH ^* (M,\CU ^\bullet _{\sm ,M})&\HH ^* (M ,\RU _{M}[0])\arrow[l,"\HH (\epsilon )"',"\cong"]\arrow[d,"\HH (\epsilon)","\cong"']\\
\HS ^* (\Omega ^\bullet (M))\arrow[rr,"\Upsilon"]\arrow[u,"\HS (\varkappa)"] & & \HH ^* (M, \Omega ^\bullet _M)\arrow[lu,"\HH (\underline{\varkappa})"]
\end{tikzcd}
\end{equation}
The diagram without the bottom row commutes for obvious reasons.
The bottom left quadrangle commutes by Diagram~\ref{diag:kappacommutative}.
The bottom right triangle commutes by Diagram~\ref{diag:propertiesofkappa}.
We do not know if the morphism $\HH (\epsilon)$ in the middle row is an isomorphism.

Next, we consider the following diagram:
\begin{equation}\label{diag:bottomrighttriangle}
\begin{tikzcd}[column sep = 0.8in]
\HS ^* (\CS ^\bullet _{\sm} (M))\arrow[r,"\HH (\epsilon ) ^{-1}\circ\Upsilon\circ\HS (\sh )"]&\HH ^* (M, \RU _{M} [0])\arrow[d,"\HH (\epsilon)","\cong"']\\
\HS ^* (\Omega ^\bullet (M))\arrow[u,"\HS (\varkappa)"]\arrow[r,"\Upsilon"']\arrow[ru,"\Theta ^{-1}"]&\HH ^* (M, \Omega ^\bullet _M).
\end{tikzcd}
\end{equation}
By the bottom stage of the previous diagram the outer contour commutes.
The bottom triangle commutes by the definition of $\Theta$, Subsection~\ref{sssec:thetamorphism}.
The vertical map $\HH (\epsilon)$ is an isomorphism.
Therefore, the upper triangle is commutative.

Consider the following diagram:
\begin{equation*}
\begin{tikzcd}[column sep = 0.8in]
\HS ^* (\Omega ^\bullet _{B|\R})\arrow[dd,"\HS (\Omega _i)"]\arrow[rr,"\Phi _B",bend left=15]\arrow[r,"\HS (\xi _B)"]\arrow[rd,"\HS (\xi _B ^\sm)"']&\HS ^* (\CS ^\bullet _{\Lip} (\spec B))\arrow[d,"\HS (\res )"]\arrow[r,"\HH (\epsilon ) ^{-1}\circ\Upsilon\circ\HS (\sh )"']&\HH ^* (\spec B, \RU _{\spec B} [0])\arrow[dd,"\Gamma _B ^*"]\\
&\HS ^* (\CS ^\bullet _{\sm} (\spec B))\arrow[d,"\HS (\CS _\sm (\Gamma _B))"]&\\
\HS ^* (\Omega ^\bullet _{C^\infty (M)|\R})\arrow[rd,"\HS (\pi)"']&\HS ^* (\CS ^\bullet _{\sm} (M))\arrow[r,"\HH (\epsilon ) ^{-1}\circ\Upsilon\circ\HS (\sh )"']&\HH ^* (M, \RU _{M} [0])\\
&\HS ^* (\Omega ^\bullet (M))\arrow[u,"\HS (\varkappa)"]\arrow[ru,"\Theta ^{-1}"']&
.\end{tikzcd}
\end{equation*}
The upper triangle is commutative by the definition of $\Phi _B$.
The left pentagon is commutative by Lemma~\ref{lem:integrationcommutativity}.
The left triangle is commutative by Diagram~\ref{diag:lipandsmooth}.
The right pentagon is commutative by Diagram~\ref{diag:rightpentagon}.
The bottom right triangle commutes by Diagram~\ref{diag:bottomrighttriangle} and $\Theta$ is an isomorphism.
Thus, the whole diagram is commutative.
Therefore, the claim follows.
\end{proof}

\subsection{The calculation of $\Psi _{C^\infty (M)}$}

In the previous section we have constructed a morphism $\Psi _A:\HS ^* (\Omega ^\bullet _{A|\R})\to \HH ^* (X,\RU _{X} [0])$ for every topological space $X$ and subalgebra $A\subset C(X)$.
In case $X=M$ a smooth manifold and $A=C^\infty (M)$ we would like to calculate this map explicitly.

By passing in Lemma~\ref{lem:phiandpidiagram} to colimit over all finitely generated subalgebras $B\subset C^\infty (M)$ we get
\begin{theorem}\label{thm:compositionproposition}
The following diagram is commutative:
\begin{equation*}\label{diag:psicommutativelemma}
\begin{tikzcd}
\HS ^*(\Omega ^\bullet _{C^\infty (M)|\R})\arrow[rr,"\Psi _{C^\infty (M)}",bend left =15]\arrow[r,"\HS (\pi)"']&\HS ^*(\Omega ^\bullet (M))&\HH ^* (M,\RU _{M} [0])\arrow[l,"\Theta"].
\end{tikzcd}
\end{equation*}
\end{theorem}

\section{The composition $\HH ^* (X,\RU _X[0] )\xrightarrow{\Lambda _{\FC}} \HS ^* (\Omega ^\bullet _{\FC (X)|\R})\xrightarrow{\Psi _{\FC (X)}} \HH ^* (X, \RU _X [0])$}\label{sec:composition}

\begin{lemma}\label{lem:compositionisidentityforpolyhedra}
Take $M$ a compact smooth manifold.
Then the following diagram is commutative:
\begin{equation*}
\begin{tikzcd}
&\HS ^* (\Omega ^\bullet _{C^\infty (M)|\R} )\arrow[dr,"\Psi _{C^\infty (M)}"]&\\
\HH ^* (M,\RU _M [0] )\arrow[ur,"\Lambda _{C^\infty _M}"]\arrow[rr,"\Id"]&&\HH ^* (M,\RU _M [0] )
.\end{tikzcd}
\end{equation*}
\end{lemma}

\begin{proof}
Consider the following diagram:
\begin{equation*}
\begin{tikzcd}
\HS ^* (\Omega ^\bullet _{C^\infty (M)|\R} )\arrow[dr,"\Psi _{C^\infty (M)}"]\arrow[r,"\HS (\pi )"]&\HS ^* (\Omega ^\bullet (M))\\
\HH ^* (M,\RU _M [0] )\arrow[u,"\Lambda _{C^\infty _M}"]\arrow[r,"\Id"]&\HH ^* (M,\RU _M [0] )\arrow[u,"\Theta"']
.\end{tikzcd}
\end{equation*}
The right triangle commutes by Theorem~\ref{thm:compositionproposition}.
The outer contour is commutative by Theorem~\ref{thm:thecompositionforsmoothfunctions} .
Since $\Theta$ is an isomorphism, the left triangle also commutes.
\end{proof}

\begin{corollary}\label{cor:compositionforsmooth}
For a compact smooth manifold $M$ the following diagram is commutative:
\begin{equation*}
\begin{tikzcd}
&\HS ^* (\Omega ^\bullet _{C(M)|\R} )\arrow[dr,"\Psi _{C(M)}"]&\\
\HH ^* (M,\RU _M [0] )\arrow[ur,"\Lambda _{C_M}"]\arrow[rr,"\Id "]&&\HH ^* (M, \RU _M [0])
.\end{tikzcd}
\end{equation*}
\end{corollary}
\begin{proof}
The inclusion morphism of sheaves $C^\infty _M \hookrightarrow C _M$ allows us to consider the diagram
\begin{equation*}
\begin{tikzcd}
\HS ^* (M,\RU _M [0])\arrow[r,"\Lambda _{C_M}"]\arrow[rd,"\Id"']\arrow[rr,"\Lambda _{C^\infty _M}",bend left =15]&\HS ^* (\Omega ^\bullet _{C(M)|\R} )\arrow[d,"\Psi _{C(M)}"]&\HS ^* (\Omega ^\bullet _{C^\infty (M)|\R})\arrow[l]\arrow[ld,"\Psi _{C^\infty (M)}"]\\
&\HH ^* (M, \RU _M [0]).&
\end{tikzcd}
\end{equation*}
The outer contour is commutative by Lemma~\ref{lem:compositionisidentityforpolyhedra}.
The right triangle is commutative by Proposition~\ref{prp:functorialityofpsiinspaces}.
The upper triangle is commutative by Proposition~\ref{prp:lambdaisfunctorial}.
Hence, the left triangle is commutative.
\end{proof}

\begin{corollary}\label{cor:mainsplittingcor}
For a compact Hausdorff space $X$ the following diagram is commutative:
\begin{equation*}
\begin{tikzcd}
&\HS ^* (\Omega ^\bullet _{C(X)|\R} )\arrow[dr,"\Psi _{C(X)}"]&\\
\HH ^* (X,\RU _X [0] )\arrow[ur,"\Lambda _{C_X}"]\arrow[rr,"\Id"]&&\HH ^* (X,\RU _X [0] )
.\end{tikzcd}
\end{equation*}
\end{corollary}

\begin{proof}
Choose a cohomology class $\lambda \in \HH ^* (X,\RU _X [0])$.
We show that $(\Psi _{C(X)} \circ \Lambda _{C_X} )(\lambda ) = \lambda$.
First, by \cite[II.$5.10$]{godement1958topologie} there is a polyhedron $N$ (the geometric realization of some nerve) and a continuous map $f:X\to N$, such that $\lambda =f^* (\delta )$ for some $\delta\in \HH ^* (N,\RU _N [0] )$.
Second, there exists a compact smooth manifold (with boundary) $M$ such that $N\subset M$ and $N$ is a deformation retract of $M$ (see \cite[Theorem~$1$]{hirsch1962smooth}).
There exists $\gamma \in \HH ^* (M,\RU _M [0])$ such that $\gamma |_N = \delta$.
Consider the composition $g:X\xrightarrow{f}N\hookrightarrow M$.
We have $\lambda = g^* (\gamma )$.
Consider the diagram
\begin{equation*}
\begin{tikzcd}
\HH ^* (X,\RU _X [0]) \arrow[r,"\Lambda _{C_X}"] & \HS ^* (\Omega ^\bullet _{C(X)|\R} ) \arrow[r,"\Psi _{C(X)}"] & \HH ^* (X,\RU _{X} [0])\\
\HH ^* (M,\RU _X [0] ) \arrow[r,"\Lambda _{C_M}"]\arrow[u,"g^*"] & \HS ^* (\Omega ^\bullet _{C(M)|\R} ) \arrow[r,"\Psi _{C(M)}"]\arrow[u,"\HS (\Omega _{g^*})"] & \HH ^* (M, \RU _M [0])\arrow[u,"g^*"]
.\end{tikzcd}
\end{equation*}
This diagram is commutative by Propositions~\ref{prp:lambdaisfunctorial} and ~\ref{prp:functorialityofpsiinspaces}.
By Corollary~\ref{cor:compositionforsmooth} the equality $(\Psi _{C(X)} \circ \Lambda _{C_X} )(\lambda ) = \lambda$ follows.
\end{proof}

\begin{theorem}\label{thm:compositionforsoftsheaf}
For a compact Hausdorff space $X$ and a soft subsheaf of algebras $\FC \hookrightarrow C_X$ the following diagram is commutative
\begin{equation*}
\begin{tikzcd}
&\HS ^* (\Omega ^\bullet _{\FC (X)|\R} )\arrow[dr,"\Psi _{\FC (X)}"]&\\
\HH ^* (X,\RU _X [0] )\arrow[ur,"\Lambda _{\FC}"]\arrow[rr,"\Id"]&&\HH ^* (X,\RU _X [0] )
.\end{tikzcd}
\end{equation*}
\end{theorem}
\begin{proof}
It immediately follows from naturality of $\Lambda$ and $\Psi$ (Propositions~\ref{prp:lambdaisfunctorial} and \ref{prp:functorialityofpsiinspaces}) and Corollary~\ref{cor:mainsplittingcor}. 
\end{proof}

\section{Piecewise polynomial functions}\label{sec:piecewisepolynomialfunctions}
\subsection{Polyhedra and rectilinear maps}\label{ssec:polyhedraandmaps}
A \textit{polyhedron} $K$ is a finite set of affine simplices in $\R ^m$ such that
\begin{enumerate}
\item
for any $a\in K$ and any face $b\subset a$ we have $b\in K$;
\item
if $a,b\in K$ then $a\cap b$ is either a common face of $a$ and $b$ or empty.
\end{enumerate}
A subset $P$ of $K$ that is also a polyhedron is called a \textit{subpolyhedron} of $K$. 
Define the space $|K|\subset \R ^m$ as the union of all simplices of $K$.

We say that a function $f:|K|\to |K^\prime |$ is a \textit{rectilinear map} if for any $a\in K$ there exists $b\in K^\prime$ such that $f(a)\subset b$ and $f$ maps $a$ to $b$ affinely.
We call two rectilinear maps $f,g:|K|\to |K^\prime |$ \textit{adjacent} if for every simplex $a \in K$ the set $f(a )\cup g(a )$ is contained in a simplex of $K^\prime$.

We say that a polyhedron $P$ is a \textit{minor} of a polyhedron $K$ if $|P|\subset |K|$ and $|P|\hookrightarrow |K|$ is a rectilinear map (in other words, for any $a\in P$ there is $b\in K$ such that $a\subset b$).
We call a minor $P$ of $K$ a \textit{subdivision} if $|P|=|K|$.

We call a polyhedron $S$ a \textit{star} with the center $x\in\R ^m$ if $\lbrace x\rbrace$ is a vertex of $S$ and each maximal by inclusion simplex of $S$ has $\lbrace x\rbrace$ as a vertex.
Take a polyhedron $K$ with $x\in |K|$, then a minor $S$ of $K$ is called a \textit{star neighborhood} of $x$ if $S$ is a star with center $x$ and $x\in \Int _{|K|} |S|$.

For a polyhedron $K$ consider the algebra $\Pol (K)$ of functions $\varphi :|K|\to\R$ such that for each simplex $a\in K$  $\varphi |_a$ is a polynomial.
This algebra was considered, for example, in \cite{billera1989algebra}.
The algebra $\Pol (K)$ is contravariant with respect to rectilinear maps, in particular, if $i:|P|\hookrightarrow |K|$ is an inclusion of a minor we have the restriction homomorphism $\Pol (i):\Pol (K)\to \Pol (P)$.

\begin{lemma}\label{lem:subdivisionlemmaforpolyhedra}
Suppose $K,P_1,\dots,P_l$ are polyhedra such that $P_i$ are minors of $K$.
Then there exist subdivisions $K^\prime$ of $K$ and $P_i ^\prime$ of $P_i$ such that $P_i ^\prime$ is a subpolyhedron of $K^\prime$ for each $i$.
\end{lemma}
\begin{proof}
See \cite[Addendum~$2.12$]{rourke2012introduction}.
\end{proof}

\begin{lemma}\label{lem:polyhedralneigbourhoodofaclosedset}
Take $K$ a polyhedron, $U\subset |K|$ a set open in $|K|$ and $D\subset U$ a set closed in $|K|$.
Then there exists a minor $P$ of $K$ such that $|P|\subset U$ and $D\subset \Int _{|K|} |P|$.
\end{lemma}

\begin{definition}\label{def:piecewisepolynomialfunctions}

Let us define the sheaf of piecewise polynomial functions on a polyhedron.
\begin{enumerate}[leftmargin=*]
\item
For a polyhedron $K$ and a subset $U\subset |K|$ open in $|K|$ we call a function $s :U\to\R$ \textit{piecewise polynomial} if for each point $x\in U$ there exists a minor $K_x$ of $K$, such that $|K_x|\subset U$ with $x\in \Int _{|K|} |K_x|$ and $s|_{|K_x|}\in \Pol (K_x)$.
\item
The set $\PW (U)$ of piecewise polynomial functions on $U$ forms an algebra (use Lemma~\ref{lem:subdivisionlemmaforpolyhedra}).
\item
Take two sets $V\subset U$ open in $|K|$.
It is not hard to see that the restriction of a piecewise polynomial function $s\in \PW (U)$ to $V$ is piecewise polynomial.
Hence, the correspondence $U\mapsto \PW (U)$ defines a sheaf on $|K|$ which we denote by $\PW _K$.
\end{enumerate}
\end{definition}

\begin{proposition}\label{prp:ppolissoft}
For a polyhedron $K$ the sheaf $\PW _K$ is soft.
\end{proposition}
\begin{proof}
Take a subset $D\subset |K|$ closed in $|K|$ and an element of $\PW _K (D)$, which is represented by a set $U\subset |K|$ open in $|K|$ with $D\subset U$ and a section $s\in \PW _K (U)$.
By Lemma~\ref{lem:polyhedralneigbourhoodofaclosedset} there exists a minor $P$ of $K$ such that $D\subset \Int _{|K|} |P|$ and $|P|\subset U$.
Apply Lemma~\ref{lem:polyhedralneigbourhoodofaclosedset} again to obtain a minor $P_2$ of $K$ such that $|P|\subset \Int _{|K|} |P_2|$ and $|P_2|\subset U$.
By Lemma~\ref{lem:subdivisionlemmaforpolyhedra} there exist subdivisions $K^\prime$ of $K$, $P^\prime _2$ of $P_2$ and $P^\prime$ of $P$ such that $P^\prime$ and $P^\prime _2$ are subpolyhedra of $K^\prime$.
Take an element $t\in \Pol (K^\prime)$ such that $t|_{|P^\prime|} \equiv 1$ and $t|_{|K^\prime|-|P^\prime _2|}\equiv 0$.
The function $ts$ is a global section of $\PW _K$ and $ts$ and $s$ coincide in $\PW _K(D)$.
\end{proof}

\subsection{The maps $\Lambda _{\PW _K}$ and $\Psi _{\PW _K (|K|)|\R}$}
Here we prove that for the sheaf of piecewise polynomial functions on $K$ the maps $\Lambda _{\PW _K}$ and $\Psi _{\PW _K (|K|)|\R}$ are in fact isomorphisms.

The following notion can be found in \cite{gersthn1971homotopy}.
Two morphisms of $\R$-algebras $\varphi _0,\varphi _1:A\to B$ are called \textit{simply homotopic} if there exists a homomorphism $H:A\to B\otimes \R [t]$ such that the following diagram is commutative for $\lambda =0,1$:
\begin{equation}\label{diag:homotopicmorphismsofalgebras}
\begin{tikzcd}
A\arrow[rr,"\varphi _\lambda",bend left =25]\arrow[r,"H"']&B\otimes \R[t]\arrow[r,"t\mapsto \lambda",swap]&B
.\end{tikzcd}
\end{equation}

The following is a well known definition of homotopic morphisms of dg-algebras and can be found in \cite[II.$1$]{lehmann1977theorie}.
\begin{definition}\label{def:homotopicmorphismsofalgebras}
Two morphisms of dg-algebras $\varphi _0,\varphi _1:E\to E^\prime$ are called \textit{homotopic} if there exists a morphism $H:E\to  E^\prime\otimes \Omega ^\bullet _{\R [t]|\R}$ such that the following diagram is commutative for $\lambda =0,1$:
\begin{equation*}
\begin{tikzcd}
E\arrow[rr,"\varphi _\lambda",bend left =25]\arrow[r,"H"']&E^\prime\otimes \Omega ^\bullet _{\R [t]|\R}\arrow[r,"t\mapsto \lambda",swap]&E^\prime
.\end{tikzcd}
\end{equation*}
Here $t\mapsto \lambda$ is the dg-algebra morphism that is the identity on $E^\prime$ and sends $t$ to $\lambda$.
\end{definition}
\begin{lemma}\label{lem:homotopicmapsofdgalgebras}
Homotopic morphisms of dg-algebras induce equal maps on the cohomology groups.
\end{lemma}
\begin{proof}
See \cite[Lemma~II$.1$]{lehmann1977theorie}.
\end{proof}

\begin{proposition}\label{prp:adjecentmapsinducethesamemorphismoncohomologygroups}
Suppose the morphisms $\varphi _0,\varphi _1:A\to B$ of algebras are simply homotopic.
Then the induced morphisms $\HS (\Omega _{\varphi _\lambda}) :\HS ^* (\Omega ^\bullet _{A|\R}) \to \HS ^* (\Omega ^\bullet _{B|\R})$ for $\lambda =0,1$ are equal.
\end{proposition}

\begin{proof}
We prove that the morphisms of dg-algebras $\Omega _{\varphi _\lambda}: \Omega ^\bullet _{A|\R}\to \Omega ^\bullet _{B|\R}$ are homotopic for $\lambda =0,1$.
By Lemma~\ref{lem:homotopicmapsofdgalgebras} it will imply that the maps on the cohomology are equal.
As $\varphi _\lambda$ are simply homotopic for $\lambda =0,1$, we have a homomorphism $H:A\to B\otimes \R [t]$ such that Diagram~\ref{diag:homotopicmorphismsofalgebras} commutes.
The obvious morphisms $\Omega ^\bullet _{B|\R}\to \Omega ^\bullet_{B\otimes \R[t]|\R}$ and $\Omega ^\bullet _{\R [t]|\R}\to \Omega ^\bullet_{B\otimes \R[t]|\R}$ form the canonical morphism
$u:\Omega ^\bullet_{B|\R}\otimes \Omega ^\bullet _{\R [t]|\R}\to\Omega ^\bullet_{B\otimes \R[t]|\R}$, which is an isomorphism by \cite[Corollary~$4.2$]{KunzKahlerDifferentials}.

Consider the following commutative diagram:
\begin{equation*}
\begin{tikzcd} 
&\Omega ^\bullet _{A|\R}\arrow[ld,"u^{-1}\circ \Omega _H"']\arrow[d,"\Omega _H"]\arrow[rd,"\Omega _{\varphi _\lambda}"]&\\
\Omega ^\bullet_{B|\R}\otimes \Omega ^\bullet _{\R [t]|\R}\arrow[r,"\cong"',"u"]\arrow[rr,"t\mapsto \lambda"',bend right = 15]&\Omega ^\bullet_{B\otimes \R[t]|\R}\arrow[r,"t\mapsto \lambda"]&\Omega ^\bullet _{B|\R}
.\end{tikzcd}
\end{equation*}
We obtain that $\Omega _{\varphi _\lambda}$ are homotopic for $\lambda =0,1$.
\end{proof}

\begin{lemma}\label{lem:conjmapsinducehomotopicmaps}
Suppose $f_0,f_1:|K|\to |K^\prime |$ are two adjacent rectilinear maps of polyhedra.
Then the induced homomorphisms $f_0 ^*,f_1 ^*:\Pol (K^\prime )\to \Pol (K)$ are simply homotopic.
\end{lemma}
\begin{proof}
Consider the algebra $T$ of functions $s$ on $|K|\times [0,1]$ such that for each $a\in K$ the restriction $s|_{a\times [0,1]}$ is a polynomial function.
The map $\Pol (K) \otimes \R [t] \to T$ sending $\beta\otimes p(t)$ to the function $(x,t)\mapsto \beta (x)p(t)$ is an isomorphism.
Construct the homomorphism $H:\Pol (K^\prime )\to T$ as $H(\alpha ) (x,t):=\alpha (f_0(x)(1-t)+f_1(x)t)$.
As $f_0$ and $f_1$ are adjacent, $H(\alpha )\in T$.
Also, $H(\alpha )|_{t=\lambda} = \alpha (f_\lambda(x))=f_\lambda ^* (\alpha )$ for $\lambda =0,1$.
The needed homotopy map is the lift of the homomorphism $H$ to $\Pol (K)\otimes \R [t]$.
\end{proof}

For a dg-algebra $E$ we consider the morphism of complexes, the coaugmentation, $\epsilon :\R [0]\to E$ defined by $\epsilon (1)=1$.

\begin{corollary}[Poincar\'e lemma]\label{cor:poincarelemma}
Suppose $S$ is a star with center $x$.
Then the coaugmentation $\epsilon :\R [0]\to \Omega ^\bullet _{\Pol (S)|\R}$ is a quasi-isomorphism.
\end{corollary}

\begin{proof}
We denote by $Q$ the one-point polyhedron $\lbrace\lbrace x\rbrace\rbrace$.
Consider the rectilinear maps $\mathrm{col}:|S|\to |Q|$ and $i:|Q|\hookrightarrow |S|$.
By the definition of a star the composition $i\circ \mathrm{col}$ is adjacent to the identity $\Id :|S|\to |S|$.
Consider the following diagram:
\begin{equation*}
\begin{tikzcd}[column sep = 0.8in]
&\HS ^*(\R [0])\arrow[ld,"\HS (\epsilon )"']\arrow[d,"\HS (\epsilon )"',"\cong"]\arrow[rd,"\HS (\epsilon )"]&\\
\HS ^*(\Omega ^\bullet _{\Pol (S)|\R})\arrow[r,"\HS (\Omega _{\Pol (i)})"]\arrow[rr,"\Id",bend right = 15]&\HS ^*(\Omega ^\bullet _{\Pol (Q)|\R})\arrow[r,"\HS (\Omega _{\Pol (\mathrm{col})})"]&\HS ^*(\Omega ^\bullet _{\Pol (S)|\R})
\end{tikzcd}
\end{equation*}

By Proposition~\ref{prp:adjecentmapsinducethesamemorphismoncohomologygroups} and Lemma~\ref{lem:conjmapsinducehomotopicmaps} the bottom triangle is commutative.
The rest of the diagram commutes for obvious reasons.
One can easily see that the morphism of complexes $\epsilon :\R [0]\to \Omega ^\bullet _{\Pol (Q)|\R}$ is an isomorphism.
Hence, $\HS (\epsilon ):\HS ^* (\R [0])\to \HS ^*(\Omega ^\bullet _{\Pol (K)|\R})$ is an isomorphism.
\end{proof}

\begin{lemma}\label{lem:subdivisionforstars}
Take $K$ a polyhedron and $x\in |K|$.
Suppose $S_0,S_1$ are star neighborhoods of $x$ in $K$, then there exists a star neighborhood $S^\prime$ of $x$ in $K$ which is a minor of $S_i$ for $i=0,1$.
\end{lemma}
\begin{proof}
Follows from Lemma~\ref{lem:subdivisionlemmaforpolyhedra}.
\end{proof}

Consider a polyhedron $K$ and a point $x\in |K|$.
Let $\Sigma (x)$ be the set of all star neighborhoods $S$ of $x$ in $K$.
Define a partial order on $\Sigma (x)$ as $S\leq S^\prime$ if $S^\prime$ is a minor of $S$.
By Lemma~\ref{lem:subdivisionforstars} the set $\Sigma (x)$ is directed.
Each star neighborhood $S$ of $x$ in $K$ gives rise to a homomorphism $\Pol (S)\to (\PW _K )_x$.
By taking the colimit over the directed set $\Sigma (x)$ we obtain the homomorphism
$$\chi :\underset{S\in\Sigma (x)}{\colim} \Pol (S) \to (\PW _K)_x.$$

\begin{lemma}\label{lem:stalkofpwarecolimitsofstarneighborhoods}
The map $\chi$ is an isomorphism.
\end{lemma}

\begin{proof}
For injectivity, take a star neighborhood $S$ of $x$ in $K$ and $t\in\Pol (S)$ such that $\chi (t)=0$.
There exists a set $U$ open in $|K|$ with $x\in U$ such that $t|_U=0$ as a function.
Take a star neighborhood $S^\prime$ of $x$ such that $|S^\prime|\subset U$ and $S^\prime$ is a minor of $S$.
We observe that $t|_{|S^\prime|}=(t|_U)|_{|S^\prime|}=0$ and, hence, $t$ is zero in $\colim _{S\in\Sigma (x)} \PW _K(S)$.

For surjectivity, take a germ $g\in (\PW _K)_x$, a set $U\subset |K|$ open in $|K|$ and a section $\tilde{g}\in \PW _K (U)$ such that $\tilde{g}_x = g$.
As $\tilde{g}\in \PW _K (U)$ there is a minor $P$ of $K$ such that $|P|\subset U$, $x\in \Int _{|K|} |P|$ and $\tilde{g}|_{|P|}\in \Pol (P)$.
Take a star neighborhood $S$ of $x$ in $P$. 
Then $\chi$ maps $\tilde{g}|_{|S|}$ to $g$.
\end{proof}

\begin{proposition}\label{prp:coaugmentationmapforpolyhedraisaqi}
For a polyhedron $K$ the coaugmentation
$$\epsilon :\RU _{|K|}[0] \to \OmegaU ^\bullet _{\PW _K|\R}$$ is a quasi-isomorphism.
\end{proposition}

\begin{proof}
We prove that $\epsilon$ is a quasi-isomorphism on stalks.
We have
$$(\RU _{|K|}[0])_x\to (\OmegaU ^\bullet _{\PW _K|\R})_x=(\Omega ^\bullet _{\PW _K|\R})_x=\Omega ^\bullet _{(\PW _K) _x|\R}\cong\underset{S\in\Sigma (x)}{\colim}\Omega ^\bullet _{\Pol (S)|\R}.$$
The last isomorphism follows from Lemma~\ref{lem:stalkofpwarecolimitsofstarneighborhoods}.
By Corollary~\ref{cor:poincarelemma} the morphism $(\RU _{|K|}[0])_x\to\colim _{S\in\Sigma (x)} \Omega ^\bullet _{\Pol (S)|\R}$ is a quasi-isomorphism.
\end{proof}

\begin{theorem}\label{thm:lambdaandpsiforpiecewisepolynomialfunctions}
The maps
$$\Lambda _{\PW _K}:\HH ^* (|K|,\RU _{|K|} [0])\to \HS ^* (\Omega ^\bullet _{\PW _K (|K|)|\R})$$
and
$$\Psi _{\PW _K (|K|)} :\HS ^* (\Omega ^\bullet _{\PW _K (|K|)|\R})\to\HH ^* (|K|,\RU _{|K|} [0])$$
are isomorphisms.
\end{theorem}
\begin{proof}
By Theorem~\ref{thm:compositionforsoftsheaf} it is enough to prove that $\Lambda _{\PW _K}$ is an isomorphism.
Recall that $\Lambda _{\PW _K}$ is defined as the diagonal map in the following diagram
\begin{equation*}
\begin{tikzcd}
\HH ^* ({|K|},\RU _{|K|} [0] )\arrow[r,"\HH (\epsilon )"]\arrow[ddr,"\Lambda _{\PW _K}"',dashed] & \HH ^* ({|K|},\OmegaU ^\bullet _{\PW _K|\R})\\
 & \HS ^* (\OmegaU ^\bullet _{\PW _K|\R} ({|K|})) \arrow[u,"\Upsilon"',"\cong"]\\
 & \HS ^* (\Omega ^\bullet _{\PW _K ({|K|})|\R}) \arrow[u,"\cong","\HS (\sh )"'].
\end{tikzcd}
\end{equation*}
By Proposition~\ref{prp:coaugmentationmapforpolyhedraisaqi} the map $\HH (\epsilon )$ is an isomorphism, hence, $\Lambda _{\PW _K}$ is an isomorphism.
\end{proof}

\section{When are the maps $\Lambda$ and $\Psi$ isomorphisms?}\label{sec:themapspsiandlambdaforsoftfunctionalgebras}

\subsection{$\HS ^0 (\Omega ^\bullet _{A|k})$ for a function algebra $A$}\label{ssec:H0forfunctionalgebra}

Let $k$ be a field.

\begin{definition}
We call %the elements $a_1,\dots ,a_n$ of a $k$-algebra $A$ algebraically independent over $k$ if the homomorphism
%$\varphi :k[x_1,\dots ,x_n] \to A$, $\varphi (x_i)=a_i$, is injective.
the elements $a_1,\dots ,a_n $ of a $k$-algebra $A$ algebraically independent over $k$ if there is no non-zero polynomial $q$ over $k$ such that $q(a_1,\dots ,a_n) =0$.
We call an element $a$ algebraic over $k$ if there is a non-zero polynomial $p$ over $k$ such that $p(a) =0$.
\end{definition}

\begin{lemma}\label{lem:reducinglemma}
Let $A$ be a $k$-algebra.
Suppose the elements $a_1,\dots ,a_n$ of $A$ are algebraically independent over $k$.
Then there exists a prime ideal $\pp \subset A$ such that the images of $a_1,\dots ,a_n$ in $A/\pp$ are algebraically independent over $k$.
\end{lemma}
%We only need the case of $B$ being finitely generated.

\begin{proof}
\iffalse
The nilradical of $B$ is the intersection of a finite number of prime ideals: $\mathrm{nil} (B)=\cap _{i=1} ^m \pp _i$.
Denote by $\pi _i :B\to B/\pp _i$ the canonical projection.
Assume the elements $\pi _i(b_1),\dots ,\pi _i(b_n)$ are algebraically dependent in $B/\pp_i$ for all $i$.
Then there exist some non-zero polynomials $q_i\in k[x_1,\dots ,x_n]$ such that $q_i(\pi _i(b_1),\dots ,\pi _i(b_n))=0$ in $B/\pp _i$.
Set $\tilde{q} := \prod _{i=1} ^m q_i$.
Then $\tilde{q} (\pi _i(b_1),\dots ,\pi _i(b_n))=0$ in all $B/\pp _i$, in other words, $\tilde{q} (b_1,\dots ,b_n)$ lies in the nilradical.
Hence, for some natural $l$ we have $\tilde{q} ^l (b_1,\dots ,b_n) = 0$ in $B$.
Set $q:=\tilde{q} ^l$ and obtain that $b_1,\dots ,b_n$ are algebraically dependent.
\fi
See \cite[Proposition~$6$]{gomez1990number}.
\end{proof}

\begin{lemma}\label{lem:FieldElementCriteria}
Suppose $K/k$ is a field extension for $k$ of characteristic $0$.
Consider the differential $d:K\to \Omega ^1_{K|k}$.
Then for $x\in K$ we have $dx = 0$ iff $x$ is algebraic over $k$.
\end{lemma}
\begin{proof}
See \cite[$\S 26$]{matsumura1989commutative}.
\end{proof}

\begin{lemma}\label{lem:ZeroCohomologyOfGeneralAlgebras}
Consider a $k$-algebra $A$ for $k$ of characteristic $0$ and take $a\in A$ such that $da=0$ in $\Omega ^1 _{A|k}$.
Then $a$ is algebraic.
\end{lemma}

This is \cite[Proposition~$7$]{gomez1990number} where the assumption on characteristic being zero is implicit.

\begin{proof}
Assume $a$ is transcendental.
By Lemma~\ref{lem:reducinglemma} there exists a prime $\pp\subset A$ such that the image $\bar{a}$ of $a$ in $A/\pp$ is transcendental.
Hence, $\bar{a}$ is transcendental in the field of fractions $\mathrm{Frac} (A/\pp )$.
By Lemma~\ref{lem:FieldElementCriteria} $d\bar{a}\neq 0$ in $\Omega ^1 _{\mathrm{Frac} (A/\pp )|k}$, which is a contradiction.

\iffalse
We have $\Omega ^1 _{A|k} =\colim \Omega ^1 _{B|k}$, where $B$ runs through all finitely generated subalgebras of $A$ (see Lemma~\ref{lem:kahlerformspreservecolimits}). 
We can find a finitely generated subalgebra $B\subset A$ such that $a\in B$ and $d_B a=0$ for the differential $d_B :B\to \Omega ^1 _{B|k}$.
The nilradical of $B$ can be written as a finite intersection of prime ideals, that is, there exists a finite set of prime ideals $\lbrace \pp _1,\dots ,\pp _m \rbrace$ of $B$ such that $\mathrm{nil} (B)=\cap _{i=1} ^m \pp _i$.
Take the image $\bar{a} _i$ of $a$ in $B/\pp _i$ for each prime ideal $\pp _i$ in the above decomposition.
The algebra $B/\pp _i$ embeds in its field of fractions $\mathrm{Frac} (B/\pp _i)$.
By functoriality of $\Omega ^1$ the element $\bar{a} _i$ maps to zero under the map $d:\mathrm{Frac} (B/\pp _i)\to \Omega ^1 _{\mathrm{Frac} (B/\pp _i)|k}$.
By Lemma~\ref{lem:FieldElementCriteria}, $\bar{a} _i$ is algebraic over $k$.
Denote by $p_i$ the minimal polynomial of $\bar{a} _i$ over $k$ and set $\tilde{p}:=\prod _i p_i$.
Finally, the element $\tilde{p}(a)$ lies in every prime ideal $\pp _i$ and, hence, some power $\tilde{p} ^l(a)$ is zero.
Put $p:=\tilde{p}^l$.
\fi
%See \cite[Proposition~$7$]{gomez1990number}.
\end{proof}

We call a subalgebra $A$ of $\mathrm{Maps} (X,k)$ a \textit{function algebra} on a set $X$.
\begin{theorem}\label{thm:finitenumberofvalues}
Take $A$ a function algebra on $X$ and $a\in A$.
Then $da=0$ in $\Omega ^1 _{A|k}$ if and only if $a$ takes a finite number of values.
\end{theorem}
This theorem is proved in \cite[proof of Theorem~$8$]{gomez1990number}.
We give a proof, which is different in the ``if'' direction.
A related result is \cite[Proposition~3]{OsbornDerivations}.
\begin{proof}
Assume first $da=0$.
By Lemma~\ref{lem:ZeroCohomologyOfGeneralAlgebras} there exists a non-zero polynomial $p$ such that $p(a)=0$.
Hence, $a$ takes a finite number of values.

Conversely, assume $a$ attains distinct values $r_1,\dots,r_m \in k$ and consider the polynomial $q(t)=(t-r_1)\dots(t-r_m)$.
Clearly $q(a)=0$ and $\gcd (q,q^\prime )=1$.
Hence, there exist polynomials $h_0,h_1\in k[t]$ such that $h_0q+h_1q^\prime =1$.
Substitute $a$ into this equality and obtain $h_1(a)q^\prime (a)=1$; hence, $q^\prime(a)$ is invertible.
Take the equality $q(a)=0$ and apply $d$ to both sides, we get $q^\prime(a)da=0$ and subsequently $da=0$.
\end{proof}

\begin{corollary}\label{cor:locallyconstantfunctions}
Take $A$ a function algebra on a space $X$.
Then $\HS ^0 (\Omega ^\bullet _{A|k})$ is the vector space of functions $a\in A$ that take a finite number of values.
\end{corollary}
A related result is \cite[Proposition~$5$]{OsbornDerivations}.

\begin{proposition}\label{prp:lambdaisisomorphismindegree0}
For a soft subsheaf of algebras $\FC \hookrightarrow C_X$ on a compact Hausdorff space $X$ the map
$\Lambda _\FC:\HH ^0 (X,\RU [0])\to \HS ^0 (\Omega ^\bullet _{\FC (X)|\R})$
is an isomorphism.
\end{proposition}
\begin{proof}
By Corollary~\ref{cor:locallyconstantfunctions} the space $\HS ^0 (\Omega ^\bullet _{\FC (X)|\R} )$ consists of functions $f\in \FC (X)$ that take a finite number of values.
As $\FC $ is a subsheaf of $C_X$ and $X$ is compact, the functions in $\FC (X)$ that take a finite number of values are exactly the locally constant functions from $\FC (X)$.
The morphism of sheaves of algebras $\RU _X \to \FC$ induces a homomorphism $\RU _X (X)\to \FC (X) = \Omega ^0 _{\FC (X)|\R}$.
By above, this homomorphism can be extended to the morphism of complexes $\tilde{\epsilon}:\RU _X(X)[0]\to \Omega ^\bullet _{\FC (X)|\R}$ and $\HS (\tilde{\epsilon})$ is an isomorphism in degree $0$.

The coaugmentation $\epsilon:\RU_X[0]\to\OmegaU^\bullet _{\FC |\R}$ induces a morphism of complexes $\epsilon (X):\RU_X (X)[0]\to \OmegaU^\bullet _{\FC |\R} (X)$.
The following diagram commutes:
\begin{equation*}
\begin{tikzcd}[column sep=0.7in]
\RU_X (X)[0]\arrow[d,"\epsilon (X)"]\arrow[rd,"\tilde{\epsilon}"]&\\
\OmegaU ^\bullet _{\FC |\R} (X)&\Omega ^\bullet _{\FC (X)|\R}\arrow[l,"\sh"].
\end{tikzcd}
\end{equation*}

Consider the diagram:
\begin{equation*}
\begin{tikzcd}
\HH ^0 (X,\RU _X [0])\arrow[d,"\HH (\epsilon )"']\arrow[rrd,"\Lambda _\FC",bend left = 40]&\HS ^0 (\RU _X (X)[0])\arrow[l,"\Upsilon","\cong"']\arrow[d,"\HS (\epsilon (X))"]\arrow[rd,"\HS (\tilde{\epsilon })","\cong"']&\\
\HH ^0 (X,\OmegaU ^\bullet _{\FC |\R})&\HS ^0 (\OmegaU ^\bullet _{\FC |\R} (X))\arrow[l,"\Upsilon","\cong"']&\HS ^0 (\Omega ^\bullet _{\FC (X)|\R})\arrow[l,"\HS (\sh )","\cong"'].
\end{tikzcd}
\end{equation*}
The outer contour commutes by the definition of $\Lambda _\FC$.
The right triangle commutes by the above diagram.
The left square obviously commutes.
Hence, the whole diagram is commutative.
The upper map $\Upsilon$ is an isomorphism by a property of $\Upsilon$ (see Subsection~\ref{ssec:sheaves}).
Therefore, $\Lambda _\FC$ is an isomorphism.
\end{proof}

It turns out that the group $\HS ^0 (\Omega ^\bullet _{C(X)|\R})$ behaves just as we expect.
\begin{corollary}\label{cor:psiisisoindegree0}
For a soft subsheaf of algebras $\FC$ of $C_X$ on a compact Hausdorff space $X$ the map
$$\Psi _{\FC (X)}:\HS ^0 (\Omega ^\bullet _{\FC (X)|\R}) \to \HH ^0 (X,\RU [0])$$
is an isomorphism.
\end{corollary}
\begin{proof}
It follows directly from Proposition~\ref{prp:lambdaisisomorphismindegree0} and Theorem~\ref{thm:compositionforsoftsheaf}.
\end{proof}

\subsection{$\Lambda$ and $\Psi$ are not isomorphisms in general}\label{ssec:lambdaandpsiarenotiso}

Take $k$ a field.

\begin{lemma}\label{lem:tracelemma}
Suppose $k$ is of characteristic $0$.
Suppose $F/k$ is a field extension and $L/F$ is a finite field extension.
Then the map
$$\HS (\Omega _i) :\HS ^* (\Omega ^\bullet _{F|k})\to \HS ^* (\Omega ^\bullet _{L|k})$$
induced by the inclusion $i:F\hookrightarrow L$ is injective.
\end{lemma}

\begin{proof}
There is a morphism of complexes $\Sigma :\Omega ^\bullet _{L|k}\to \Omega ^\bullet _{F|k} $ with the following properties (see \cite[$\S 16$]{KunzKahlerDifferentials}):
\begin{enumerate}
\item
if we consider $\Omega ^\bullet _{L|k}$ as a $\Omega ^\bullet _{F|k}$-module, then $\Sigma$ is $\Omega ^\bullet _{F|k}$-linear;
\item
the restriction of $\Sigma$ to the elements of degree $0$ coincides with the trace $\sigma :L\to F$.
\end{enumerate}
Therefore, the composition
$$\Omega ^\bullet _{F|k} \xrightarrow{\Omega _i} \Omega ^\bullet _{L|k}\xrightarrow{\Sigma}\Omega ^\bullet _{F|k}$$
is the multiplication by $\sigma (1)=[L:F]$.
As $k$ is of characteristic zero the map $\HS (\Omega _i) $ is an injection.
\end{proof}

\iffalse
\begin{lemma}\label{lem:exactnesscolimit}
Suppose a $k$-algebra $A$ is a directed colimit of its subalgebras $A_\alpha$, $\alpha\in\Lambda$:
$$A\cong \underset{\alpha\in\Lambda}{\colim}A_{\alpha\in \Lambda}$$
with all subalgebras $A_\alpha$ containing a subalgebra $D\subset A$.
Suppose a form $\omega\in \Omega ^n _{D|k}$ is exact in $\Omega ^n _{A|k}$.
Then there exists $\alpha\in\Lambda$ such that $\omega$ is exact $\Omega ^n _{A_\alpha|k}$.
\end{lemma}

\begin{proof}
As the functors $\Omega ^\bullet$ and $\HS ^*$ preserve colimits we have
$$\HS ^* (\Omega ^\bullet _{A|k})\cong \underset{\alpha\in\Lambda}{\colim}\HS ^* (\Omega ^\bullet _{A_{\alpha\in \Lambda}|k}).$$
As the class $[\omega]\in \HS ^n (\Omega ^\bullet _{D|k})$ is zero in $\HS ^n (\Omega ^\bullet _{A|k})$ it is zero in some $\HS ^n (\Omega ^\bullet _{A_\alpha |k})$.
\end{proof}
\fi
\begin{lemma}\label{lem:transctracelemma}
Let $K/k$ be a field extension, where $K$ is an infinite field, and $K(\Gamma )/K$, where $\Gamma = \lbrace \gamma _1,\dots ,\gamma _l\rbrace$, be a purely transcendental extension.
Then the map
$$\HS (\Omega _i) :\HS ^* (\Omega ^\bullet _{K|k})\to \HS ^* (\Omega ^\bullet _{K(\Gamma)|k}),$$
induced by the inclusion $i:K\hookrightarrow K(\Gamma)$, is injective.
\end{lemma}
\begin{proof}
The field $K(\Gamma)$ is the colimit over finite sets $S\subset K[\Gamma]-\lbrace 0\rbrace$ of the localizations of $K[\Gamma]$:
$$K(\Gamma)\cong \underset{S\subset K[\Gamma]-\lbrace 0\rbrace}{\colim}K[\Gamma][S^{-1}].$$
Therefore, it is enough to prove that the map
$$\HS (\Omega _{i_S}):\HS ^* (\Omega ^\bullet _{K|k})\to \HS ^* (\Omega ^\bullet _{K[\Gamma][S^{-1}]|k})$$
induced by the inclusion $i_S:K\hookrightarrow K[\Gamma][S^{-1}]$ is injective for each $S$.

There exists a point $(\bar{\gamma}_1,\dots,\bar{\gamma}_l)\in K^{l}$ which is not a zero of any element of $S$.
Consider the homomorphism $K[\Gamma]\to K$, $\gamma _j\mapsto \bar{\gamma} _j$.
By the universal property of localization it can be extended to a homomorphism $\tau :K[\Gamma][S^{-1}]\to K$ such that $\tau |_K$ is the identity.
Therefore, the composition
$$\HS ^* (\Omega ^\bullet _{K|k})\xrightarrow{\HS (\Omega _{i_S})} \HS ^* (\Omega ^\bullet _{K[\Gamma][S^{-1}]|k})\xrightarrow{\HS (\Omega _\tau )} \HS ^* (\Omega ^\bullet _{K|k})$$
is the identity map, which suffices.
\end{proof}

\begin{lemma}\label{lem:technicalexactnesslemma}
Suppose the characteristic of $k$ is zero.
The equation
\begin{equation}\label{eq:formequation}
\sum _{i=1} ^n x_i (\partial /\partial x_i) F_i = c
\end{equation}
where $c\in k$ and $c\neq 0$, has no solutions in rational functions $F_i\in k(x_1,\dots ,x_n)$.
\end{lemma}

\begin{proof}
We proceed by induction on $n$.
For $n=0$ the claim is obvious.
Take $l\in \N \cup \lbrace 0\rbrace$ such that, for each $i$, $x_1$ is not a factor in the denominators of $x_1 ^l F_i$.
Multiply the equation~(\ref{eq:formequation}) by $x_1 ^l$ and apply $\left(\partial /\partial x_1\right) ^l$ to both sides of the equation.
We get
$$\sum _{i=1} ^n \left( \partial /\partial x_1\right) ^l(x_1 ^l x_i (\partial /\partial x_i) F_i) = l!c.$$
Notice that the operators $x(\partial /\partial x)$ and $\left( \partial /\partial x\right) ^l x^l$ commute.
That can be easily checked by induction.
Hence, the operators $x_i (\partial /\partial x_i)$ and $\left( \partial /\partial x _j\right) ^l x^l _j$ commute for any $i$ and $j$. 
We get
$$\sum _{i=1} ^n x_i (\partial /\partial x_i)\left( \partial /\partial x_1\right) ^l (x_1 ^l  F_i) = l!c.$$
Put $G_i := \left( \partial /\partial x_1 \right) ^l (x_1 ^l  F_i)$.
We obtain
$$\sum _{i=1} ^n x_i (\partial /\partial x_i) G_i = l!c.$$
But none of $G_i$ has $x_1$ as a factor in the denominator so we can substitute $x_1:=0$ in the above equation and obtain
$$\sum _{i=2} ^n x_i (\partial /\partial x_i) (G_i |_{x_1 =0}) = l!c.$$
By the induction hypothesis this equation has no solutions in rational functions.
\end{proof}

\begin{lemma}\label{lem:exactnesslemma}
Suppose the characteristic of $k$ is zero.
The closed form
$$\omega = \dfrac{dx_1}{x_1} \wedge \dots \wedge \dfrac{dx_n}{x_n} \in \Omega ^n _{k(x_1,\dots ,x_n)|k}$$
is not exact.
\end{lemma}

\begin{proof}
Suppose there exists $\eta\in\Omega ^{n-1} _{k(x_1,\dots ,x_n)|k}$ such that $d\eta =\omega$.
Write $\eta$ as
$$\eta = \sum _{i=1} ^n F_i dx_1 \wedge\dots\wedge\widehat{dx_i} \wedge\dots\wedge dx_n$$
where $F_i\in k(x_1,\dots ,x_n)$.
Then
$$d\eta = \left[ \sum _{i=1} ^n (-1)^{i+1}(\partial /\partial x_i) F_i\right] dx_1 \wedge\dots\wedge dx_n.$$
The equality $d\eta =\omega$ then takes the form
$$\left[ \sum _{i=1} ^n (-1)^{i+1}(\partial /\partial x_i) F_i\right] dx_1 \wedge\dots\wedge dx_n = \dfrac{1}{x_1\dots x_n} dx_1 \wedge\dots\wedge dx_n.$$

The vector $dx_1\wedge\dots\wedge dx_n$ of the vector space $\Omega ^n _{k(x_1,\dots ,x_n)|k}$ over the field $k(x_1,\dots ,x_n)$  is not zero \cite[p. $201$]{matsumura1989commutative}.
Hence,
$$\sum _{i=1} ^n (-1)^{i+1}(\partial /\partial x_i) F_i = \dfrac{1}{x_1 \dots x_n}.$$
Put $G_i := (-1)^{i+1} x_1\dots \hat{x} _i \dots x_n F_i$.
Then the above equation takes the form
$$\sum _{i=1} ^n x_i (\partial /\partial x_i) G_i = 1.$$
By Lemma~\ref{lem:technicalexactnesslemma} this equation has no solutions in rational functions.
\end{proof}

\begin{theorem}\label{thm:maintheorem}
Suppose $k$ is of characteristic $0$.
Take $A$ a $k$-algebra and a set of invertible elements $a_1,\dots ,a_n$ in $A$ that are algebraically independent over $k$.
Then the closed form
$$\omega =\dfrac{da_1}{a_1}\wedge\dots\wedge\dfrac{da_n}{a_n} \in \Omega ^n _{A|k}$$
is not exact.
\end{theorem}

\begin{proof}
By Lemma~\ref{lem:kahlerformspreservecolimits} it is enough to consider $A$ being finitely generated.
By Lemma~\ref{lem:reducinglemma} there is a prime ideal $\pp\subset A$ such that the elements $a_1,\dots ,a_n$ are still algebraically independent in $A/\pp$.
Consider the field of fractions $\mathrm{Frac} (A/\pp)$ which is finitely generated over $k$ as a field.
The elements $a_1,\dots ,a_n$ are algebraically independent in $\mathrm{Frac} (A/\pp)$.

Take a finite transcendental basis $\hat{\Gamma} =\lbrace a_1,\dots ,a_n,\gamma _1,\dots ,\gamma _l \rbrace$ of $\mathrm{Frac} (A/\pp)$ over $k$.
The field $\mathrm{Frac} (A/\pp)$ is a finite extension of $k(\hat{\Gamma} )$.
By Lemma~\ref{lem:tracelemma} and Lemma~\ref{lem:transctracelemma} the composition $k(a_1,\dots , a_n)\to k(\hat{\Gamma} )\to \mathrm{Frac} (A/\pp)$ induces an injection on the de Rham cohomology.

The image of $\omega$ in $\Omega ^n _{\mathrm{Frac} (A/\pp)|k}$ coincides with the image of the closed form
$$\omega _0 =\dfrac{da_1}{a_1}\wedge\dots\wedge \dfrac{da_n}{a_n}\in \Omega ^n _{k(a_1,\dots , a_n)|k},$$
which is not exact by Lemma~\ref{lem:exactnesslemma}.
\end{proof}

\begin{theorem}\label{thm:psinotinjective}
Let $X$ be a topological space and $f\in C(X)$ take an infinite number of distinct values.
Take a subalgebra $A\subset C(X)$ such that $e^{\lambda f}\in A$ for all $\lambda \in \R$.
Then for each $n\geq 1$ $\HS ^n (\Omega ^\bullet _{A|\R})\neq 0$, moreover, the map
$$\Psi _A :\HS ^n (\Omega ^\bullet _{A|R})\to \HH ^n (X,\RU _X [0])$$
is not injective.
\end{theorem}

\begin{proof}
Choose a set of linearly independent over $\mathbb{Q}$ numbers $\lambda _1,\dots ,\lambda _n \in \R$ and consider the functions $a_i :=e^{\lambda _i f(x)}\in A$.
These functions are algebraically independent and invertible, hence, the closed form
$$\omega =\dfrac{da_1}{a_1}\wedge\dots\wedge\dfrac{da_n}{a_n}\in \Omega ^n _{A|\R}$$
is not exact by Theorem~\ref{thm:maintheorem}.

For the second part, we prove that $\Psi _A ([\omega ]) = 0$.
Denote by $i:A\hookrightarrow C(X)$ the inclusion.
Consider the algebra $E\subset C(\R )$ generated as an algebra by the functions $e^{\lambda t}\in C(\R )$.
We have the commutative diagram
\begin{equation*}
\begin{tikzcd}
C(X)&A\arrow[l,hook']\\
C(\R )\arrow[u,"f^*"]&E \arrow[l,hook']\arrow[u,"\varphi"'],
\end{tikzcd}
\end{equation*}
where the homomorphism $\varphi$ is induced by the homomorphism $f^*$.
Then the following diagram is commutative by Proposition~\ref{prp:functorialityofpsiinspaces}:
\begin{equation*}
\begin{tikzcd}
\HS ^n (\Omega ^\bullet _{A|\R}) \arrow[r,"\Psi _{A}"] & \HH ^n (X,\RU _{X} [0])\\
\HS ^n (\Omega ^\bullet _{E|\R}) \arrow[u,"\HS(\Omega _{\varphi})"]\arrow[r,"\Psi _{E}"] & \HH ^n (\R,\RU _{\R} [0]) \rlap{=0.}\arrow[u,"f^*"']
\end{tikzcd}
\end{equation*}
The equality $\HH ^n (\R,\RU _{\R} [0]) =0$ is well known.

Consider the functions $\tilde{a}_i :=e^{\lambda _i t}\in E$ and the closed form
$$\tilde{\omega} =\dfrac{d\tilde{a}_1}{\tilde{a}_1}\wedge\dots\wedge\dfrac{d\tilde{a}_n}{\tilde{a}_n}\in \Omega ^n _{E|\R}.$$
We have $\omega = \Omega _{\varphi} (\tilde{\omega})$.
By the commutativity of the above diagram, $\Psi _{A} ([\omega])=0$.
Since $[\omega]\neq 0$, the map $\Psi _{A}$ is not injective.
\end{proof}

\begin{corollary}\label{cor:lambdaisnotsurjective}
Suppose $X$ is a compact Hausdorff space and $\FC$ is a soft subsheaf of $C_X$ such that $\FC (X)$ satisfies the conditions imposed on the algebra $A$ in Theorem~\ref{thm:psinotinjective}.
Then
$\Lambda _\FC :\HH ^n(X,\RU _X [0])\to \HS ^n (\Omega ^\bullet _{\FC|\R})$
is not surjective.
\end{corollary}
\begin{proof}
By Theorem~\ref{thm:compositionforsoftsheaf} the composition $\Psi _{\FC (X)}\circ \Lambda _\FC = \Id$.
By Theorem~\ref{thm:psinotinjective} $\Psi _{\FC (X)}$ is not injective and, hence, $\Lambda _\FC$ is not surjective.
\end{proof}

As the first example one can consider a smooth manifold $M$ of positive dimension and $A=C^\infty (M)$.
Then $\HS ^n (\Omega ^\bullet _{C^\infty (M)|\R})\neq 0$ for $n\geq 1$.
Also, one can consider an infinite compact Hausdorff space $X$ and $A=C(X)$.
Then $\HS ^n (\Omega ^\bullet _{C(X)|\R})\neq 0$ for $n\geq 1$.
\printbibliography

\Addresses
\end{document}